\documentclass[12pt]{amsart}
\usepackage{amssymb}
\usepackage{amsfonts}
\usepackage{amssymb,latexsym}
\usepackage{enumerate}
\usepackage{mathrsfs}
\makeatletter
\@namedef{subjclassname@2010}{%
  \textup{2010} Mathematics Subject Classification}
\makeatother

  [2002/01/19 v2.2g %
    AMS font definitions%
  ]
\DeclareFontFamily{U}{euf}{}
\DeclareFontShape{U}{euf}{m}{n}{%
  <5><6><7><8><9>gen*eufm%
  <10><10.95><12><14.4><17.28><20.74><24.88>eufm10%
  }{}
\DeclareFontShape{U}{euf}{b}{n}{%
  <5><6><7><8><9>gen*eufb%
  <10><10.95><12><14.4><17.28><20.74><24.88>eufb10%
  }{}

  [2002/01/19 v2.2g %
    AMS font definitions%
  ]
\DeclareFontFamily{U}{msb}{}
\DeclareFontShape{U}{msb}{m}{n}{%
  <5><6><7><8><9>gen*msbm%
  <10><10.95><12><14.4><17.28><20.74><24.88>msbm10%
  }{}

  [2002/01/19 v2.2g %
    AMS font definitions%
  ]
\DeclareFontFamily{U}{msa}{}
\DeclareFontShape{U}{msa}{m}{n}{%
  <5><6><7><8><9>gen*msam%
  <10><10.95><12><14.4><17.28><20.74><24.88>msam10%
  }{}

\newtheorem{theorem}{Theorem}[section]
\newtheorem{lemma}[theorem]{Lemma}
\newtheorem{proposition}[theorem]{Proposition}
\newtheorem{corollary}[theorem]{Corollary}

\theoremstyle{definition}
\newtheorem{definition}[theorem]{Definition}

\theoremstyle{remark}
\newtheorem{remark}[theorem]{Remark}

\numberwithin{equation}{section} 

\def\C{\mathbb C_p}
\def\BZ{\mathbb Z}
\def\Z{\mathbb Z_p}
\def\Q{\mathbb Q_p}
\def\C{\mathbb C_p}
\def\BZ{\mathbb Z}
\def\Z{\mathbb Z_p}
\def\Q{\mathbb Q_p}
\def\psum{\sideset{}{^{(p)}}\sum}

\begin{document}

\title[On $p$-adic Hurwitz-type Euler zeta functions]
{On $p$-adic Hurwitz-type Euler zeta functions}

\author{Min-Soo Kim and Su Hu}


\address{Department of Mathematical Sciences, Korea Advanced Institute of Science and Technology (KAIST), 373-1 Guseong-dong, Yuseong-gu, Daejeon 305-701, South Korea}
\email{minsookim@kaist.ac.kr}


\address{Department of Mathematical Sciences, Korea Advanced Institute of Science and Technology (KAIST), 373-1 Guseong-dong, Yuseong-gu, Daejeon 305-701, South Korea}
\email{hus04@mails.tsinghua.edu.cn, husu@kaist.ac.kr}

%\thanks{This work is supported by Kyungnam University Foundation Grant, 2008.}

%    General info


\begin{abstract}
The definition for the $p$-adic Hurwitz-type Euler zeta functions
has been given by using the fermionic $p$-adic integral on $\mathbb
Z_p$. By computing  the values of this kind of $p$-adic zeta
function at negative integers, we show that it interpolates the
Euler polynomials $p$-adically. Many properties are provided for the
$p$-adic Hurwitz-type Euler zeta functions, including the convergent
Laurent series expansion, the distribution formula, the functional
equation, the reflection formula, the derivative formula, the
$p$-adic Raabe formula and so on. The definition for the $p$-adic
Euler $L$-functions has also been given by using the $p$-adic
Hurwitz-type Euler zeta functions.
\end{abstract}

\maketitle

\def\C{\mathbb C_p}
\def\BZ{\mathbb Z}
\def\Z{\mathbb Z_p}
\def\Q{\mathbb Q_p}

% we give a new approach to

\section{Introduction}

Throughout this paper, we use the following notations.
\begin{equation*}
\begin{aligned}
\qquad \mathbb{C}  ~~~&- ~\textrm{the field of complex numbers}.\\
\qquad p  ~~~&- ~\textrm{an odd rational prime number}. \\
\qquad\mathbb{Z}_p  ~~~&- ~\textrm{the ring of $p$-adic integers}. \\
\qquad\mathbb{Q}_p~~~&- ~\textrm{the field of fractions of}~\mathbb Z_p.\\
\qquad\mathbb C_p ~~~&- ~\textrm{the completion of a fixed algebraic closure}~\overline{\mathbb Q}_p~ \textrm{of}~\mathbb Q.\\
\qquad C\mathbb{Z}_{p}  ~~~&- ~\mathbb{Q}_{p}\backslash\mathbb{Z}_{p}. \\
\qquad v_p ~~~&- ~\textrm{the $p$-adic valuation of}~\mathbb C_p~\textrm{normalized so that}~ |p|_p=p^{-v_p(p)}=p^{-1}.\\
\qquad \Z^\times~~~&- ~\textrm{the group of $p$-adic units}.
\end{aligned}
\end{equation*}

We say that $f:\mathbb Z_p\rightarrow\mathbb C_p$ is uniformly
differentiable function at a point $a\in\mathbb Z_p,$ and we write
$f\in UD(\mathbb Z_p),$ if the difference quotients $\Phi_f:\mathbb
Z_p\times\mathbb Z_p\rightarrow\mathbb C_p$ such that
\begin{equation}\label{UD}
\Phi_f(x,y)=\frac{f(x)-f(y)}{x-y}
\end{equation}
have a limit $f'(a)$ as $(x,y)\rightarrow(a,a),x\neq y$ (see
\cite[p.\,221]{Ro}).

Cohen and Friedman \cite{CF} constructed the $p$-adic
analogue for Hurwitz zeta functions, and Raabe-type formulas for the
$p$-adic gamma and zeta functions from Volkenborn integrals
satisfying the modified difference equation. Here the Volkenborn
integral of a function $f:\Z\to\C$ with $f\in UD(\Z)$ is defined by
\begin{equation}\label{vol-int}
\int_{\mathbb
Z_p}f(x)dx=\lim_{N\to\infty}\frac1{p^N}\sum_{x=0}^{p^N-1}f(x)
\end{equation}
(cf. \cite[p.\,264]{Ro}). This integral was introduced by Volkenborn
\cite{Vo} and he also investigated many important properties of
$p$-adic valued functions defined on the $p$-adic domain (see
\cite{Vo,Vo1}). Recently, Tangedal and Young \cite{TP} defined
$p$-adic multiple zeta  and log gamma functions by using multiple
Volkenborn integrals, and developed many of their properties.

As shown by Cohen, the $p$-adic functions with nice properties are
powerful tools for studying many results of classical number theory
in a straightforward manner, for instance strengthenings of almost
all the arithmetic results on Bernoulli and Euler numbers, see Cohen
\cite[Chapter 11]{Co2} or the accounts in Iwasawa \cite{Iw}, Koblitz
\cite{Ko}, Lang \cite{La}, Murty \cite{Ra}, Washington \cite{Wa}.

The fermionic $p$-adic integral $I_{-1}(f)$ on $\mathbb Z_p$ is
defined by
\begin{equation}\label{-q-e}
I_{-1}(f)=\int_{\mathbb Z_p}f(a)d\mu_{-1}(a)
=\lim_{N\rightarrow\infty}\sum_{a=0}^{p^N-1}f(a)(-1)^a.
\end{equation}
where $f\in UD(\Z).$ This integral was  introduced by Kim \cite{TK}
in order to derives useful formulas involving the Euler numbers and
polynomials. It has also been defined independently by Shiratani and
Yamamoto \cite{Shi} in order to interpolate the Euler numbers
$p$-adically. Osipov \cite{Osipov} gave a new proof of the
existence of the Kubota-Leopoldt $p$-adic zeta function by using the
integral representation
\begin{equation}\label{-q-e-ep}
I_{\varepsilon}(f)=\int_{\mathbb Z_p}f(a)d\mu_{\varepsilon}(a)
=\lim_{N\rightarrow\infty}\sum_{a=0}^{p^{lN}-1}f(a)\varepsilon^a.
\end{equation}
where $f\in UD(\Z),\varepsilon^k=1,\varepsilon\neq1,(k,p)=1,$ and
$k\mid(p^l-1).$ Note that when $k=2,$ $I_{\varepsilon}(f)$ is the
fermionic $p$-adic integral $I_{-1}(f)$ on $\mathbb Z_p.$ Recently,
the fermionic $p$-adic integral $I_{-1}(f)$ on $\mathbb Z_p$   is
used by  the first author to give a brief proof of Stein's classical
result on Euler numbers modulo power of two \cite{MSK}, it has also
been used by Ma\"iga \cite{Ma} to give some new identities and
congruences concerning Euler numbers and polynomials.

In 1749, Euler gave a paper to the Berlin Academy entitled
\textit{Remarques sur un beau rapport entre les s\'eries des
puissances tant directes que r\'eciproques.} In this paper, he
studied
$$\phi(s)=\sum_{n=1}^{\infty}\frac{(-1)^{n}}{n^{s}}$$
(see \cite[p. 1081]{Ayoub}).

For $s\in\mathbb{C}$ and Re$(s) > 0$, the Euler zeta function and
the Hurwitz-type Euler zeta function are defined by
$$\zeta_{E}(s)=2\sum_{n=1}^{\infty}\frac{(-1)^{n}}{n^{s}},\quad\textrm{and}\quad
\zeta_{E}(s,x)=2\sum_{n=0}^{\infty}\frac{(-1)^{n}}{(n+x)^{s}}$$
respectively (see \cite{Ayoub,KKW,TK,TK08,Murty}). Notice that the
Euler zeta functions can be analytically continued to  the whole
complex plane, and these zeta functions have the values of the Euler
numbers or the Euler polynomials at negative integers.

In this paper, by using the fermionic $p$-adic integral $I_{-1}(f)$
on $\mathbb Z_p$ we define the $p$-adic Hurwitz-type Euler zeta
functions following Cohen's approach in \cite[Chapter 11]{Co2} and
Tangedal-Young's approach in \cite{TP}.

Our paper is organized as follows.

In section 2, we recall the relationship between the ferminoic
$p$-adic integral $I_{-1}(f)$ on $\mathbb{Z}_{p}$ and the Euler
numbers and polynomials.
And we also prove those fundamental formulas for the ferminoic $p$-adic
integral $I_{-1}(f)$ on $\Z.$ These three formulas will be used to get
the main results of this paper.

In section 3, we give the definition for the $p$-adic Hurwitz-type
Euler zeta function for $x\in
\mathbb{C}_{p}\backslash\mathbb{Z}_{p}$. By computing the values of
this kind of $p$-adic zeta function at negative integers, we show
that it interpolates the Euler polynomials $p$-adically (see
Theorem \ref{p-E-zeta-val} below). We also obtain many properties
for $p$-adic Hurwitz-type Euler zeta function for  $x\in
\mathbb{C}_{p}\backslash\mathbb{Z}_{p}$, including the convergent
Laurent series expansion, the distribution formula, the functional
equation, the derivative formula, the $p$-adic Raabe formula and so
on.

In section 4, we give the definition for  the $p$-adic Hurwitz-type
Euler zeta function for $x\in \mathbb{Z}_{p}$ using characters
modulo $p^{v}$. We also obtain many  properties for the $p$-adic
Hurwitz-type Euler zeta function  for $x\in \mathbb{Z}_{p}$,
including the distribution formula, the functional equation, the
reflection formula, the derivative formula, the $p$-adic Raabe
formula and so on. By using  the $p$-adic Hurwitz-type Euler zeta
function, we also gave a new definition for the $p$-adic Euler
$L$-function for characters modulo $p^{v}$. We show that in this
case the definition is equivalent to the first author's previous
definition in \cite{MSK2}. In \cite{MSK2}, the first author
proposed a construction of $p$-adic Euler $L$-function for Dirichlet
characters with odd conductors using Kubota-Leopoldt's approach. We
also show that the $p$-adic Hurwitz-type Euler zeta function can be
represented as the $p$-adic Euler $L$-functions using the power
series expansion under certain conditions.

In section 5, we define the $p$-adic Euler $L$-functions for any
Dirichlet characters by using the $p$-adic Hurwitz-type Euler zeta
functions. By computing the values of the $p$-adic Euler
$L$-functions at negative integers, we show that for Dirichlet
characters with odd conductor, this definition is equivalent to the
first author's previous definition in \cite{MSK2}. We also study the
behavior of $p$-adic Euler $L$-functions at positive integers. We
show that most of the results in Section 11.3.3 of Cohen's
book \cite{Co2} are also established if we replace the generalized
Bernoulli numbers with the generalized Euler numbers.

Since Euler numbers has the connections with class numbers of
imaginary quadratic fields due to Carlitz \cite{Ca} and class
numbers of general cyclic quartic fields due to Xianke
Zhang \cite{xzhang}, we believe some results in this paper may have
potential applications in algebraic number theory.

\section{Euler numbers and polynomials in $p$-adic analysis}

Let $X\subset\C$ be an arbitrary subset closed under $x\rightarrow
x+a$ for $a\in\Z$ and $x\in X.$ In particular, $X$ could be
$\C\setminus\Z, \Q\setminus\Z$ or $\Z.$ Suppose $f:X\rightarrow\C$
is uniformly differentiable on $X,$ so that for fixed $x\in X$ the
function $a\rightarrow f(x+a)$ is uniformly differentiable on $\Z.$
Let $\Delta$ be the difference operator defined by $(\Delta
f)(a)=f(a+1)-f(a)$ and put $(\nabla f)(a)=f(a)-f(a-1).$ The
following three properties of the fermionic $p$-adic integral
$I_{-1}(f)$ on $\Z$ can be directly derived from the definition:
\begin{equation}\label{de-1-int-ex-1}
\int_{\mathbb Z_p}f(x+a)d\mu_{-1}(a)=2f(x-1)-\int_{\mathbb
Z_p}(f_1)^{-1}(x+a)d\mu_{-1}(a) ;
\end{equation}
\begin{equation}\label{de-1-int-ex-2}
\int_{\mathbb Z_p}f(x+a)d\mu_{-1}(a)=f(x)-\frac{1}{2}\int_{\mathbb
Z_p}(\Delta f)(x+a)d\mu_{-1}(a) ;
\end{equation}
\begin{equation}\label{de-1-int-ex-3}
\int_{\mathbb Z_p}f(x+a)d\mu_{-1}(a)=f(x-1)+\frac{1}{2}\int_{\mathbb
Z_p}(\nabla f)(x+a)d\mu_{-1}(a).
\end{equation}

In the $p$-adic theory (\ref{de-1-int-ex-1})-(\ref{de-1-int-ex-3})
are the characterization integral equations of the Euler numbers and
polynomials: If we put $f(a)=e^{at}$ in
(\ref{de-1-int-ex-1})-(\ref{de-1-int-ex-3}), we get
\begin{equation}\label{def-e-poly-int}
\int_{\Z} e^{at}d\mu_{-1}(a)+\int_{\Z}
e^{(a-1)t}d\mu_{-1}(a)=2e^{-t},
\end{equation}
whence we may immediately deduce the following
\begin{equation}\label{def-e-poly}
e^{xt}\int_{\Z} e^{at}d\mu_{-1}(a)=\sum_{m=0}^\infty
E_n(x)\frac{t^n}{n!},
\end{equation}
where $t\in\C$ such that $|t|_p<p^{-1/(p-1)}$ and $E_n(x)$ are the
Euler polynomials. Therefore by (\ref{def-e-poly}), we get
\begin{equation}\label{int-Em}
\int_{\Z} (x+a)^nd\mu_{-1}(a)=E_n(x).
\end{equation}
If we put $f(a)=e^{(2a+1)t}$ with $t\in\C$ such that
$|t|_p<p^{-1/(p-1)}$ in (\ref{de-1-int-ex-1})-(\ref{de-1-int-ex-3}),
we get
$$\int_{\Z} e^{(2(a+1)+1)t}d\mu_{-1}(a)+\int_{\Z} e^{(2a+1)t}d\mu_{-1}(a)=2e^t,$$
whence we may immediately deduce the following
\begin{equation}\label{def-e-numbers-eq}
\int_{\Z} e^{(2a+1)t}d\mu_{-1}(a)=\sum_{m=0}^\infty
E_n\frac{t^n}{n!},
\end{equation}
where $E_m$ are the Euler numbers. By (\ref{int-Em}) and
(\ref{def-e-numbers-eq}), we have the identity
$$E_n=\int_{\Z} (2a+1)^nd\mu_{-1}(a)$$
(see \cite[Theorem 2.5]{MSK}). Using (\ref{int-Em}), this can also
be written
\begin{equation}\label{Eu-nu-pol}\begin{aligned}
E_n=2^m\int_{\Z}
\left(a+\frac12\right)^nd\mu_{-1}(a)=2^n\sum_{k=0}^n\binom
nk\left(\frac12\right)^{n-k}E_k(0)
\end{aligned}\end{equation}
and
\begin{equation}\label{Eu-pol-nu}\begin{aligned}
E_n(0)&=2^{-n}\int_{\Z}
(2a+1-1)^nd\mu_{-1}(a)\\&=2^{-n}\sum_{k=0}^n\binom nk(-1)^{n-k}E_k.
\end{aligned}\end{equation}
Taking $f(a)=e^{-(a+1)t}$ in
(\ref{de-1-int-ex-1})-(\ref{de-1-int-ex-3}), we find
\begin{equation}\label{int-eq-shi}
\int_{\Z} e^{-(a+2)t}d\mu_{-1}(a)+\int_{\Z}
e^{-(a+1)t}d\mu_{-1}(a)=2e^{-t}.
\end{equation}
It is easy to see from (\ref{def-e-poly-int}) that
(\ref{int-eq-shi}) satisfies
\begin{equation}\label{int-eq-shi-gen}
\int_{\Z}
e^{-(a+1)t}d\mu_{-1}(a)=\frac{2e^{-t}}{e^{-t}+1}=\frac{2}{e^{t}+1}=\int_{\Z}
e^{at}d\mu_{-1}(a).
\end{equation}
Combining (\ref{def-e-poly}) and (\ref{int-eq-shi-gen}), we
calculate
$$\begin{aligned}
\sum_{n=0}^\infty\int_{\Z} (x+a)^nd\mu_{-1}(a)\frac{(-t)^n}{n!}&=e^{(1-x)t}\int_{\Z} e^{-(a+1)t}d\mu_{-1}(a) \\
&=e^{(1-x)t}\int_{\Z} e^{at}d\mu_{-1}(a) \\
&=\sum_{n=0}^\infty \int_{\Z} (1-x+a)^nd\mu_{-1}(a)\frac{t^n}{n!},
\end{aligned}$$
so that
\begin{equation}\label{e-poly-symm}
(-1)^n\int_{\Z} (x+a)^nd\mu_{-1}(a)=\int_{\Z} (1-x+a)^nd\mu_{-1}(a)
\end{equation}
or equivalently
\begin{equation}\label{e-poly-symm-eq}
E_n(x)=(-1)^nE_n(1-x),
\end{equation}
which amounts to $E_n(1/2)=0$ for odd $n.$ From (\ref{int-Em}),
(\ref{Eu-nu-pol}) and (\ref{e-poly-symm-eq}), we see that
$E_{2n+1}=2^{2n+1}E_{2n+1}(1/2)=0$ for $n\geq0.$ Therefore we have
the results as follows:

\begin{proposition}\label{E-np-pro-11}
For $n\geq 0$
\begin{itemize}
\item[\rm(1)] $I_{-1}((x+a)^n)=\int_{\Z} (x+a)^n d\mu_{-1}(a)=E_n(x).$
\item[\rm(2)] $I_{-1}((2a+1)^n)=\int_{\Z} (2a+1)^n d\mu_{-1}(a)=E_n.$
\item[\rm(3)] $E_n=2^n\sum_{k=0}^n\binom nk\left(\frac12\right)^{n-k}E_k(0).$
\item[\rm(4)] $E_n(0)=2^{-n}\sum_{k=0}^n\binom nk(-1)^{n-k}E_k.$
\item[\rm(5)] $E_n(x)=(-1)^nE_n(1-x).$
\item[\rm(6)] $E_{2n+1}=0.$
\end{itemize}
\end{proposition}

\begin{theorem}\label{E-np-thm}
For any $f\in UD(\Z),$ the fermionic $p$-adic integral on $\mathbb Z_p$ satisfies the following
three properties:

\begin{itemize}
\item[\rm(1)] $\int_{\Z}f(a+1)d\mu_{-1}(a)+\int_{\Z}f(a)d\mu_{-1}(a)=2f(0).$
\item[\rm(2)] $\int_{\Z}f(a+1)d\mu_{-1}(a)=\int_{\Z}f(-a)d\mu_{-1}(a).$
\item[\rm(3)] $\int_{\Z}f(a)d\mu_{-1}(a)=\sum_{j=0}^{m-1}(-1)^j\int_{\Z}f(j+ma)d\mu_{-1}(a),$
where $m$ is an odd integer.
\end{itemize}
\end{theorem}
\begin{proof}
(1) From (\ref{-q-e}), we have
$$
\begin{aligned}
\int_{\Z}f(a+1)d\mu_{-1}(a)&=\lim_{N\to\infty}\sum_{a=0}^{p^N-1}f(a+1)(-1)^a \\
&=-\lim_{N\to\infty}\sum_{a=1}^{p^N}f(a)(-1)^a \\
&=-\lim_{N\to\infty}\left(\sum_{a=0}^{p^N-1}f(a)(-1)^a-f(0)-f(p^N) \right)\\
&=-\int_{\Z}f(a)d\mu_{-1}(a)+2f(0),
\end{aligned}
$$
which proves (1).

(2) Consider
$$
\begin{aligned}
&\quad\sum_{a=0}^{p^N-1}f(a+1)(-1)^a -\sum_{a=0}^{p^N-1}f(-a)(-1)^a  \\
&=f(1)-f(2)+\cdots+f(p^N) \\
&\quad-(f(0)-f(-1)+\cdots+f(1-p^N)) \\
&=(f(p^N)-f(0))+(-f(p^N-1)+f(-1)) \\
&\quad+(f(p^N-2)-f(-2))+\cdots+(f(1)-f(1-p^N)).
\end{aligned}
$$
Then we have
\begin{equation}\label{pf2.2-1}
\begin{aligned}
&\quad\biggl|\sum_{a=0}^{p^N-1}f(a+1)(-1)^a -\sum_{a=0}^{p^N-1}f(-a)(-1)^a\biggl|_p  \\
&\leq \max_{0\leq j\leq p^N-1}|f(p^N-j)-f(-j)|_p.
\end{aligned}
\end{equation}
Since $f$ is continuous, for any $\varepsilon>0,$ there exists a positive integer $N_0,$ such that
for any $N\geq N_0$ and any $x\in\mathbb Z_p,$ we obtain
\begin{equation}\label{pf2.2-2}
|f(p^N-x)-f(-x)|_p\leq\varepsilon.
\end{equation}
From (\ref{pf2.2-1}) and (\ref{pf2.2-2}), for any $N\geq N_0,$ we have
$$\biggl|\sum_{a=0}^{p^N-1}f(a+1)(-1)^a -\sum_{a=0}^{p^N-1}f(-a)(-1)^a\biggl|_p\leq\varepsilon,$$
thus
$$\lim_{N\to\infty}\sum_{a=0}^{p^N-1}f(a+1)(-1)^a=\lim_{N\to\infty}\sum_{a=0}^{p^N-1}f(-a)(-1)^a,$$
giving (2).

(3) Note that if $m$ is odd, then we have $(-1)^{j+ma}=(-1)^{j+a}$ and
$$\begin{aligned}
\sum_{j=0}^{m-1}(-1)^j\sum_{a=0}^{p^N-1}f(j+ma)(-1)^a
&=\sum_{j=0}^{m-1}\sum_{a=0}^{p^N-1}f(j+ma)(-1)^{j+ma} \\
&=\sum_{a'=0}^{mp^N-1}f(a')(-1)^{a'}.
\end{aligned}$$
Letting $N\to\infty$ in the above equality, we get our result.
\end{proof}

\section{The $p$-adic Hurwitz-type Euler zeta functions}\label{Results}

Before passing to the $p$-adic theory, we take a quick look at the
complex analytic Hurwitz-type Euler zeta function $\zeta_{E}(s,x).$

\begin{definition}
We define the Hurwitz-type Euler zeta function $\zeta_E(s,x)$ for
$x\in\mathbb R_{>0}$ and $s\in\mathbb C$ with Re$(s)>0$ by
$$\zeta_{E}(s,x)=2\sum_{n=0}^\infty\frac{(-1)^n}{(n+x)^s}.$$
It is known that $\zeta_{E}(s,x)$ can be extended to the whole
$s$-plane by means of contour integral (see \cite{TK08}).
\end{definition}

For $x\in\mathbb R_{>0}$ and $m\geq0,$ the formula
\begin{equation}\label{E-zeta-ne}
\zeta_{E}(-m,x)=E_m(x)
\end{equation}
holds (cf. \cite{TK08}). This identity (\ref{E-zeta-ne}) obtained
earlier in the course of Euler's proof for the values at
non-positive integer arguments of the Riemann zeta function
$\zeta(s)=\sum_{n=1}^\infty1/n^s$ (cf. \cite{Ayoub,KKW}). From
(\ref{int-Em}) and (\ref{E-zeta-ne}), we can write
\begin{equation}\label{E-zeta-int}
\int_{\Z}(x+a)^md\mu_{-1}(a)=E_m(x)=\zeta_{E}(-m,x).
\end{equation}
This is used to construct a $p$-adic Hurwitz-type Euler zeta
function.

Given $x\in\Z,p\nmid x$ and $p>2,$ there exists a unique $(p-1)$th
root of unity $\omega(x)\in\mathbb Z_p$ such that
$$x\equiv\omega(x)\pmod{p},$$
where $\omega$ is the Teichm\"uller character. Let $\langle
x\rangle=\omega^{-1}(x)x,$ so $\langle x\rangle \equiv1\pmod p.$

Next we  define the projection  $\langle x\rangle$ for all $x\in
\mathbb{C}_p^{\times}$, as was done by Kashio \cite{Kashio} and by
Tangedal-Young in \cite{TP}.

Fixing an embedding of $\overline{\mathbb{Q}}$ into $\mathbb{C}_{p}$.
$p^{\mathbb{Q}}$ denotes the image in $\mathbb{C}_{p}^{\times}$  of
the set of positive real rational powers of $p$ under this
embedding, $\mu$ denotes the group of roots of unity in
$\mathbb{C}_{p}^{\times}$ of order not divisible by $p$. For
$x\in\mathbb{C}_{p}$, $|x|_{p}=1$, there exists a unique elements
$\hat{x}\in\mu$ such that $|x-\hat{x}|_{p} < 1$ (called the
Teichm\"uller representative of $x$); it may be defined by
$\hat{x}=\lim_{n\to\infty}x^{p^{n!}}$. We extend this definition to
$x\in\mathbb{C}_{p}^{\times}$ by
\begin{equation} \hat{x}:=(\widehat{x/p^{v_{p}}(x)}),\end{equation}
that is, we define $\hat{x}=\hat{u}$ if $x=p^{r}u$ with $p^{r}\in
p^{\mathbb{Q}}$ and $|u|_{p}=1$, then we define the function
$\langle\cdot\rangle$ on $\mathbb{C}_{p}^{\times}$ by $$\langle
x\rangle=p^{-v_{p}(x)}x/\hat{x},$$ and we also define
$\omega_v(\cdot)$ on $\mathbb{C}_{p}^{\times}$ by
$$\omega_v(x)=\frac{x}{\langle x\rangle}.$$  From this we get an
internal product decomposition of multiplicative groups
\begin{equation}
\mathbb{C}_{p}^{\times}\simeq p^{\mathbb{Q}}\times\mu\times D
\end{equation}
where $D=\{x\in\mathbb{C}_{p}: |x-1|_{p} < 1\},$ given by
\begin{equation}
x=p^{v_{p}(x)}\cdot\hat{x}\cdot\langle x\rangle\mapsto
(p^{v_{p}(x)},\hat{x},\langle x\rangle).\end{equation} As remarked
by Tangedal and Young in \cite{TP}, this decomposition of
$\mathbb{C}_{p}^{\times}$ depends on the choice of embedding of
$\overline{\mathbb{Q}}$ into $\mathbb{C}_{p}$; the projections
$p^{v_{p}(x)},\hat{x},\langle x\rangle$ are uniquely determined up
to roots of unity. However for $x\in\mathbb{Q}_{p}^{\times}$ the
projections $p^{v_{p}(x)},\hat{x},\langle x\rangle$ are uniquely
determined and do not depend on the choice of the embedding. Notice
that the projections $x\mapsto p^{v_{p}(x)}$ and $x\mapsto \hat{x}$
are constant on discs of the form $\{x\in\mathbb{C}_{p}:|x-y|_{p} <
|y|_{p}\}$ and therefore have derivative zero whereas the
projections $x\mapsto\langle x\rangle$ has derivative
$\frac{d}{dx}\langle x\rangle=\langle x\rangle/x$.

For $x\in\mathbb{C}_{p}^{\times}$ and $s\in\mathbb{C}_{p}$ we define
$\langle x\rangle^{s}$ (\cite[p. 141]{SC}) by
\begin{equation} \langle
x\rangle^{s}=\sum_{n=0}^{\infty}\binom{s}{n}(\langle
x\rangle-1)^{n}\end{equation} when this sum is convergence.

\begin{proposition}[see Tangedal and Young \cite{TP}]\label{analytic} For any
 $x\in\mathbb{C}_{p}^{\times}$ the function $s\mapsto\langle
 x\rangle^{s}$ is a $C^{\infty}$ function of $s$ on
 $\mathbb{Z}_{p}$ and is analytic on a disc of positive radius about
 $s=0$; on this disc it is locally analytic as a function of $x$ and
 independent of the choice made to define the $\langle
 \cdot\rangle$ function. If $x$ lies in a finite extension $K$ of
 $\mathbb{Q}_{p}$ whose ramification index over $\mathbb{Q}_{p}$ is
 less than $p-1$ then  $s\mapsto\langle
 x\rangle^{s}$  is analytic for $|s|_{p} <
 |\pi|_{p}^{-1}p^{-1/(p-1)}$, where $(\pi)$ is the maximal ideal of
 the ring of integers $O_{K}$ of $K$. If $s\in\mathbb{Z}_{p}$, the
 function $s\mapsto\langle
 x\rangle^{s}$ is an analytic function of $x$  on any disc of the form $\{x\in\mathbb{C}_{p}:|x-y|_{p} <
|y|_{p}\}$.
\end{proposition}

\begin{definition}\label{p-E-zeta}
For $x\in \mathbb{C}_{p}\backslash \mathbb{Z}_{p}$, we define the
$p$-adic Hurwitz-type Euler zeta function $\zeta_{p,E}(s,x)$ by the
equivalent formulas
$$\zeta_{p,E}(s,x)=\int_{\Z}\langle x+a\rangle^{1-s}d\mu_{-1}(a).$$
\end{definition}

\begin{theorem}
For any  choice of $x\in\mathbb{C}_{p}\backslash \mathbb{Z}_{p}$ the
function $\zeta_{p,E}(s,x)$ is a $C^{\infty}$ function of $s$ on
$\mathbb{Z}_{p}$, and is an analytic function of $s$ on a disc of
positive radius about $s=0$; on this disc it is locally analytic as
a function of $x$ and independent of the choice made to define the
$\langle\cdot\rangle$ function. If $x$ is so chosen to lie in a
finite extension $K$ of $\mathbb{Q}_{p}$ whose ramification index
over $\mathbb{Q}_{p}$ is less than $p-1$ then $\zeta_{p,E}(s,x)$ is
analytic for $|s|_{p} <
 |\pi|_{p}^{-1}p^{-1/(p-1)}$. If $s\in\mathbb{Z}_{p}$, the function
 $\zeta_{p,E}(s,x)$ is locally analytic as a function of $x$ on
 $\mathbb{C}_{p}\backslash\mathbb{Z}_{p}$.
\end{theorem}
\begin{remark}Tangedal and Young proved the similar results for
$p$-adic multiple zeta functions (see \cite[Theorem 3.1]{TP}).
\end{remark}
\begin{proof} This follows from Proposition \ref{analytic} and
definition of $\zeta_{p,E}(s,x)$.\end{proof}

\begin{theorem}\label{p-analytic}
Suppose $x\in \mathbb{C}_{p}$ and $|x|_{p} > 1 $. Then there
is an identity of analytic functions
\begin{equation*}\begin{aligned}
\zeta_{p,E}(s,x)&=\int_{\Z}\langle x+a\rangle^{1-s}d\mu_{-1}(a)
\\&=\langle x\rangle^{1-s}\sum_{i=0}^\infty\binom{1-s}iE_i(0)\frac{1}{x^i}
\end{aligned}\end{equation*}
on a disc of positive radius about $s=0.$
If in addition  $x$ is so chosen to lie in a finite extension $K$ of
$\mathbb{Q}_{p}$ whose ramification index over $\mathbb{Q}_{p}$ is
less than $p-1$, then  this formula is valid for $s$ in
$\mathbb{C}_{p}$ such that $|s|_{p} <
 |\pi|_{p}^{-1}p^{-1/(p-1)}$, where $(\pi)$ is the maximal ideal of
 the ring of integers $O_{K}$ of $K$.
\end{theorem}

\begin{remark}Tangedal and Young proved the similar results for
$p$-adic multiple zeta functions (see \cite[Theorem 4.1]{TP}). We
follow their methods to prove this theorem.
\end{remark}
\begin{proof}
Under the stated hypotheses, for all $a\in\mathbb{Z}_{p}$, we use
Proposition \ref{analytic} to write
\begin{equation}\label{analyticnew}
\langle x+a\rangle^{1-s}=\langle
x\rangle^{1-s}\sum_{i=0}^\infty\binom{1-s}{i}a^i\frac{1}{x^i}
\end{equation}
as an identity of analytic functions. From
Proposition \ref{E-np-pro-11} (1) we have
\begin{equation*}\int_{\Z} a^{n}d\mu_{-1}(a)=E_n(0).\end{equation*}
Then we integrate (\ref{analyticnew}) with respect to $a$ and obtain
\begin{equation*}\begin{aligned}\zeta_{p,E}(s,x)&=\int_{\Z}\langle x+a\rangle^{1-s}d\mu_{-1}(a)
\\&=\langle x\rangle^{1-s}\sum_{i=0}^\infty\binom{1-s}iE_i(0)\frac{1}{x^i}.
\end{aligned}\end{equation*}
We obtain the final statement.
\end{proof}

\begin{theorem}\label{p-E-zeta-val}
Suppose that $x\in \C$ and $|x|_p>1.$
\begin{enumerate}
\item[\rm(1)] For $m\in\mathbb Z\setminus\{0\},$
$$\zeta_{p,E}(1+m,x)=\omega_v^m(x)\int_{\Z} \frac{1}{(x+a)^m}d\mu_{-1}(a).$$
\item[\rm(2)] For $m\in\mathbb N,$
$$\zeta_{p,E}(1-m,x)=\frac1{\omega_v^{m}(x)}E_m(x)=\frac1{\omega_v^{m}(x)}\zeta_E(-m,x).$$
\end{enumerate}
\end{theorem}
\begin{proof}
(1) Note that
\begin{equation}\label{3.8-1}
\zeta_{p,E}(1+m,x)=\int_{\Z}\langle x+a\rangle^{-m}d\mu_{-1}(a).
\end{equation}
Since $x+a=\omega_v(x+a)\langle x+a\rangle$ and $|x|_p>1,|a|_p\leq1,$
we have $\omega_v(x+a)=\omega_v(x)$ and
\begin{equation}\label{3.8-2}
x+a=\omega_v(x)\langle x+a\rangle.
\end{equation}
From (\ref{3.8-1}) and (\ref{3.8-2}), we complete the proof.

(2) Note that
\begin{equation}\label{3.8-3}
\zeta_{p,E}(1-m,x)=\int_{\Z}\langle x+a\rangle^{m}d\mu_{-1}(a).
\end{equation}
From (\ref{3.8-2}) and (\ref{3.8-3}), we have
$$\zeta_{p,E}(1-m,x)=\omega_v(x)^{-m}\int_{\Z}(x+a)^{m}d\mu_{-1}(a).$$
Using Proposition \ref{E-np-pro-11} (1) and (\ref{E-zeta-ne}),
we get our result.
\end{proof}

\begin{theorem}\label{zeta-1}
Suppose that $x\in \C\backslash\Z.$ We have
$$\zeta_{p,E}(1,x)= 1.$$
\end{theorem}
\begin{proof}
By Definition \ref{p-E-zeta}, we have
$$\zeta_{p,E}(1,x)=\int_{\Z}d\mu_{-1}(a)=\lim_{N\to\infty}\sum_{a=0}^{p^{N}-1}(-1)^{a}=1.$$
This completes the proof.
\end{proof}

\begin{theorem}\label{p-analytic-some-result23}
The function $\zeta_{p,E}(s,x)$ has the following properties.
\begin{enumerate}
\item[\rm(1)] For all $x\in \C\backslash\Z$ the function $\zeta_{p,E}(s,x)$ satisfies the functional equation
$$\zeta_{p,E}(s,x+1)+\zeta_{p,E}(s,x)=\frac{2x}{\omega_v(x)\langle x\rangle^{s}}.$$
\item[\rm(2)]
For all $x\in \mathbb{C}_{p}\backslash\mathbb{Z}_{p}$ we have the
reflection functional equation
$$\zeta_{p,E}(s,1-x)=\zeta_{p,E}(s,x).$$
\item[\rm(3)]
For all $x\in \mathbb{C}_{p}\backslash\mathbb{Z}_{p}$ and an odd positive integer $m$
we have the distribution formula
$$\zeta_{p,E}(s,mx)=\langle m\rangle^{1-s}\sum_{j=0}^{m-1}(-1)^j\zeta_{p,E}\left(s,x+\frac jm\right).$$
\item[\rm(4)]
For all $x\in\mathbb{C}_{p}$ and $|x|_{p} > 1$ we have derivative formula
$$\frac{\partial}{\partial x}\zeta_{p,E}(s,x)=\frac{1-s}{\omega_v(x)}\zeta_{p,E}(s+1,x)$$
holds if $s\in\mathbb{Z}_{p}$; this identity also holds for $|s|_{p} <
|\pi|_{p}^{-1}p^{-1/(p-1)}$ if $x$ lies in a finite extension $K$ of
$\mathbb{Q}_{p}$ whose ring of integers has maximal ideal $(\pi)$
with ramification index over $\mathbb{Q}_{p}$ less than $p-1$.
\end{enumerate}
\end{theorem}

\begin{remark} Tangedal and Young proved the similar results for $p$-adic multiple zeta functions
(see \cite[Theorem 3.2 (i), (iii), (iv), (vi)]{TP}).
\end{remark}

\begin{proof}
(1) From Theorem \ref{E-np-thm} (1), for any $f\in UD(\Z),$ we have
$$\int_{\Z}f(a+1)d\mu_{-1}(a)+\int_{\Z}f(a)d\mu_{-1}(a)=2f(0).$$
Let $f(a)=\langle x+a\rangle^{1-s}$ in the above equality, we get
$$\zeta_{p,E}(s,x+1)+\zeta_{p,E}(s,x)=2\langle x\rangle^{1-s}=\frac{2x}{\omega_v(x)\langle x\rangle^{s}},$$
since $x=\omega_v(x)\langle x\rangle.$

(2) From Theorem \ref{E-np-thm} (2), for any $f\in UD(\Z),$ we have
$$\int_{\Z}f(a+1)d\mu_{-1}(a)=\int_{\Z}f(-a)d\mu_{-1}(a).$$
Let $f(a)=\langle x-1+a\rangle^{1-s}$ in the above equality. Then
$$\begin{aligned}
\zeta_{p,E}(s,x)&=\int_{\Z}\langle x+a\rangle^{1-s}d\mu_{-1}(a)\\
&=\int_{\Z}f(a+1)d\mu_{-1}(a) \\
&=\int_{\Z}f(-a)d\mu_{-1}(a) \\
&=\int_{\Z}\langle x-1-a\rangle^{1-s}d\mu_{-1}(a) \\
&=\int_{\Z}\langle 1-x+a\rangle^{1-s}d\mu_{-1}(a) \\
&=\zeta_{p,E}(s,1-x),
\end{aligned}
$$
since $\langle x\rangle=\langle -x\rangle.$

(3) From Theorem \ref{E-np-thm} (3), for any $f\in UD(\Z),$ we have
$$\int_{\Z}f(a)d\mu_{-1}(a)=\sum_{j=0}^{m-1}(-1)^j\int_{\Z}f(j+ma)d\mu_{-1}(a).$$
Let $f(a)=\langle mx+a\rangle^{1-s}$ in the above equality. Then
$$\begin{aligned}
\zeta_{p,E}(s,mx)&=\int_{\Z}\langle mx+a\rangle^{1-s}d\mu_{-1}(a)\\
&=\int_{\Z}f(a)d\mu_{-1}(a) \\
&=\sum_{j=0}^{m-1}(-1)^j\int_{\Z}f(j+ma)d\mu_{-1}(a) \\
&=\sum_{j=0}^{m-1}(-1)^j\int_{\Z}\langle mx+j+ma\rangle^{1-s} d\mu_{-1}(a) \\
&=\langle m\rangle^{1-s}\sum_{j=0}^{m-1}(-1)^j\zeta_{p,E}\left(s,x+\frac jm\right).
\end{aligned}$$

(4) Since
$$\frac{\partial}{\partial x}\langle x+a\rangle^s=s\langle x+a\rangle^s(x+a)^{-1}
=s\frac{\langle x+a\rangle^{s-1}}{\omega_v(x)},$$
we have
$$\begin{aligned}
\frac{\partial}{\partial x}\zeta_{p,E}(s,x)
&=\frac{\partial}{\partial x}\int_{\Z}\langle x+a\rangle^{1-s}d\mu_{-1}(a)\\
&=\int_{\Z}\frac{\partial}{\partial x}\langle x+a\rangle^{1-s}d\mu_{-1}(a)\\
&=\frac{1-s}{\omega_v(x)}\int_{\Z}\langle x+a\rangle^{1-(s+1)}d\mu_{-1}(a).
\end{aligned}$$
This is complete the proof.
\end{proof}

\begin{theorem}\label{p-analytic-some-result}
If $x/u\in \C$ and $|x/u|_{p} >1$,  then
$$\zeta_{p,E}(s,x+u)=\langle
x\rangle^{1-s}\sum_{i=0}^\infty\binom{1-s}iE_i(u)\frac{1}{x^i}$$
on a disc of positive radius about $s=0.$
If in addition  $x$ is so chosen to lie in a finite extension $K$ of
$\mathbb{Q}_{p}$ whose ramification index over $\mathbb{Q}_{p}$ is
less than $p-1$, then  this formula is valid for $s$ in
$\mathbb{C}_{p}$ such that $|s|_{p} <
 |\pi|_{p}^{-1}p^{-1/(p-1)}$, where $(\pi)$ is the maximal ideal of
 the ring of integers $O_{K}$ of $K$.
\end{theorem}

\begin{proof}
Obviously, if $|x/u|_{p}>1,$ we can write $\langle
x+u\rangle=\langle x\rangle(1+u/x).$ One can check that
$$\begin{aligned}
\zeta_{p,E}(s,x+u)&=\langle x+u\rangle^{1-s}\sum_{i=0}^\infty\binom{1-s}iE_i(0)\frac1{(x+u)^{i}} \\
&=\langle x\rangle^{1-s}\sum_{i=0}^\infty\binom{1-s}iE_i(0)\frac{1}{x^i}\left(1+\frac ux\right)^{1-s-i} \\
&=\langle x\rangle^{1-s}\sum_{i=0}^\infty\binom{1-s}iE_i(0)\frac{1}{x^i}\sum_{j=0}^\infty\binom{1-s-i}ju^j\frac{1}{x^j} \\
&=\langle x\rangle^{1-s}\sum_{n=0}^\infty\sum_{i=0}^n\binom{1-s}i\binom{1-s-i}{n-i}E_i(0)u^{n-i}\frac{1}{x^n} \\
&=\langle x\rangle^{1-s}\sum_{n=0}^\infty\binom{1-s}n\frac{1}{x^n}\sum_{i=0}^n\binom niE_i(0)u^{n-i} \\
&=\langle
x\rangle^{1-s}\sum_{n=0}^\infty\binom{1-s}nE_n(u)\frac{1}{x^n}
\end{aligned}$$
and so the result is established.
\end{proof}

We now prove a Raabe formula for $\zeta_{p,E}$ (cf.
\cite[Proposition 3.1]{CF}).

\begin{theorem}\label{p-int-analytic-ft}
If  $x\in \C$ and $|x|_{p} > 1$, then
$$\int_{\Z}\zeta_{p,E}(s,x+a)d\mu_{-1}(a)=2\left(1+\frac1x\right)\zeta_{p,E}(s,x)
-\frac2{x\langle x\rangle}\zeta_{p,E}(s-1,x)$$ this identity  holds
if $s\in\mathbb{Z}_{p}$, also holds for $|s|_{p} <
|\pi|_{p}^{-1}p^{-1/(p-1)}$ if $x$ lies in a finite extension $K$ of
$\mathbb{Q}_{p}$ whose ring of integers has maximal ideal $(\pi)$
with ramification index over $\mathbb{Q}_{p}$ less than $p-1$.
\end{theorem}
\begin{proof}
Note that
$$\frac12\left(\frac{2e^{2xt}}{e^{t}+1}\right)^2=\frac{2(1-2x)e^{2xt}}{e^{t}+1}
+\frac{\text d}{\text{d}t}\left(\frac{2e^{2xt}}{e^{t}+1}\right)$$
(cf. \cite{Ge}). From (\ref{def-e-poly}) and (\ref{int-Em}), we get
the identity
\begin{equation}\label{p-int-lem}
\sum_{i=0}^n\binom
niE_i(x)E_{n-i}(x)=2((1-2x)E_n(2x)+E_{n+1}(2x)),\quad n\geq0.
\end{equation}
Taking $x=0$ in (\ref{p-int-lem}), we have
\begin{equation*}\begin{aligned}\int_{\Z}E_n(a)d\mu_{-1}(a)&=\sum_{i=0}^n\binom ni E_{n-i}(0)\int_{\Z}a^id\mu_{-1}(a)
\\&=2(E_n(0)+E_{n+1}(0)),\end{aligned}\end{equation*}
since $E_n(a)=\sum_{i=0}^n\binom ni E_{n-i}(0)a^i.$ From Theorem
\ref{p-analytic-some-result}, it follows that
\begin{equation}\label{zet-in-id}
\begin{aligned}
&\quad\int_{\Z}\zeta_{p,E}(s,x+a)d\mu_{-1}(a)\\
&=\langle x\rangle^{1-s}\sum_{i=0}^\infty\binom{1-s}i \frac{1}{x^i}\int_{\Z}E_i(a)d\mu_{-1}(a) \\
&=2\langle x\rangle^{1-s}\sum_{i=0}^\infty\binom{1-s}i \frac{1}{x^i}(E_i(0)+E_{i+1}(0))\\
&=2\zeta_{p,E}(s,x)+ 2\langle
x\rangle^{1-s}\sum_{i=0}^\infty\binom{1-s}i \frac{1}{x^i}E_{i+1}(0).
\end{aligned}
\end{equation}
Using $\frac{\partial}{\partial x},$ we rewrite the function
$\langle x\rangle^{1-s}\sum_{i=0}^\infty\binom{1-s}i
x^{-i}E_{i+1}(0)$ in the right-hand side. If $s\in\mathbb{Z}_{p}$ or
$|s|_{p} < |\pi|_{p}^{-1}p^{-1/(p-1)}$ when $x$ lies in a finite
extension $K$ of $\mathbb{Q}_{p}$ whose ring of integers has maximal
ideal $(\pi)$ with ramification index over $\mathbb{Q}_{p}$ less
than $p-1$, we obtain
$$\begin{aligned}
\frac{\partial}{\partial x}\zeta_{p,E}(s-1,x)
&=(2-s)\biggl(\frac{1}{x}\zeta_{p,E}(s-1,x)\\&\quad+\langle
x\rangle^{2-s}\sum_{i=0}^\infty\binom{1-s}i
\frac{1}{x^i}E_{i+1}(0)\biggl).
\end{aligned}$$
Hence
$$\begin{aligned}\langle x\rangle^{1-s}\sum_{i=0}^\infty\binom{1-s}i \frac{1}{x^i}E_{i+1}(0)
&=\frac{\langle x\rangle^{-1}}{2-s}\biggl(\frac{\partial}{\partial
x}\zeta_{p,E}(s-1,x)\\&\quad-\frac{2-s}{x}
\zeta_{p,E}(s-1,x)\biggl).\end{aligned}$$ Substituting this identity
in (\ref{zet-in-id}), it is easy to prove the desired result using
Theorem \ref{p-analytic-some-result23} (4).
\end{proof}

\section{The $p$-adic Hurwitz-type Euler zeta functions for $x\in\mathbb{Z}_{p}$}\label{Results2}

Let $\chi$ be a Dirichlet character modulo $p^{v}$ for some $v$. We can extend the definition of $\chi$ to $\mathbb{Z}_{p}$ as in \cite[p.\,281]{Co2}. In this section we define the $p$-adic Hurwitz-type Euler zeta functions for $x\in\mathbb{Z}_{p}$.

\begin{definition}\label{de2}
Let $\chi$ be a character modulo $p^{v}$ with $v\geq 1$. If
$x\in\mathbb{Z}_{p}$ and $s\in\mathbb{C}_{p}$ such that
$|s|_{p}<R_{p}=p^{(p-2)/(p-1)}$  we define
$$\zeta_{p,E}(\chi,s,x)=\int_{\mathbb{Z}_{p}}\chi(x+a)\langle x+a \rangle^{1-s}d\mu_{-1}(a),$$
and we will simply write $\zeta_{p,E}(s,x)$ instead of $\zeta_{p,E}(\chi_{0},s,x)$, where $\chi_{0}$ is a trivial character modulo $p^{v}$.\end{definition}

We show that this definition makes sense. First we prove the following lemma on the change of variable for fermionic $p$-adic integrals which is an analogy of Lemma 11.2.3 in \cite{Co2} on the change of variable for Volkenborn $p$-adic integrals.

\begin{lemma}
Let $\chi$ be a character modulo $p^{v}$, let $f$ be a function defined for $v_{p}(x)< -v$ such that for fixed $x$ the function $f(x+a)$ is in $UD(\mathbb Z_p)$, and set $$g(x)=\int_{\mathbb{Z}_{p}}f(x+a)d\mu_{-1}(a).$$ Then for $x\in\mathbb{Z}_{p}$ we have $$\sum_{j=0}^{p^{v}-1}\chi(x+j)g\left(\frac{x+j}{p^{v}}\right)(-1)^{j}
=\int_{\mathbb{Z}_{p}}\chi(x+a)f\left(\frac{x+a}{p^{v}}\right)d\mu_{-1}(a).$$\end{lemma}
\begin{proof} By definition, we have
\begin{equation*}\label{change}
\begin{aligned}
&\quad\sum_{j=0}^{p^{v}-1}\chi(x+j)g\left(\frac{x+j}{p^{v}}\right)(-1)^{j}
\\&=\lim_{N\to\infty}\sum_{j=0}^{p^{v}-1}\chi(x+j)(-1)^{j} \sum_{a=0}^{p^{N}-1}f\left(a+\frac{x+j}{p^{v}}\right)(-1)^{a}\\
&=\lim_{N\to\infty}\sum_{m=0}^{p^{v+N}-1}\chi(x+m)f\left(\frac{x+m}{p^{v}}\right)(-1)^{m}\\
&=\int_{\mathbb{Z}_{p}}\chi(x+a)f\left(\frac{x+a}{p^{v}}\right)d\mu_{-1}(a).
\end{aligned}
\end{equation*}
This completes the proof.
\end{proof}

\begin{corollary}\label{representation}
Definition \ref{de2} makes sense for $x\in\mathbb{Z}_{p}$ and $|s|_{p} < R_{p}$. More precisely, for any odd positive integer $M$ such that $p^{v} | M$, we
 have
 $$\zeta_{p,E}(\chi,s,x)=\sum_{j=0}^{M-1}\chi(x+j)\zeta_{p,E}\left(s,\frac{x+j}{M}\right)(-1)^{j}.$$
 \end{corollary}
\begin{proof} First, we prove this Corollary for $M=p^{v}$. Applying the above lemma to $f(x)=\langle x \rangle^{1-s}$, we have $$g(x)=\int_{\mathbb{Z}_{p}}\langle x+a \rangle^{1-s}\mu_{-1}(a)=\zeta_{p,E}(s,x),$$ thus
\begin{equation}\label{mid}
\begin{aligned}&\quad\sum_{j=0}^{p^{v}-1}\chi(x+j)\zeta_{p,E}\left(s,\frac{x+j}{p^{v}}\right)(-1)^{j}\\&= \langle p^{v}\rangle^{s-1}\int_{\mathbb{Z}_{p}}\chi(x+a)\langle x+a \rangle^{1-s}d\mu_{-1}(a)\\&=\zeta_{p,E}(\chi,s,x),
 \end{aligned}
\end{equation}
since $\langle p^{v}\rangle=1$.
For a general old positive integer $M$, we write $M=Np^{v}$ and $j=p^{v}a+b$ with $0\leq b< p^{v}$ and $0\leq a< N$, so that
\begin{equation*}
\begin{aligned}&\quad\sum_{j=0}^{M-1}\chi(x+j)\zeta_{p,E}\left(s,\frac{x+j}{M}\right)(-1)^{j}\\
&=\sum_{b=0}^{p^{v}-1}\chi(x+b)(-1)^{b}\sum_{a=0}^{N-1}\zeta_{p,E}\left(s,\frac{x+b}{Np^{v}}+\frac{a}{N}\right)(-1)^{a}\\
&=\sum_{b=0}^{p^{v}-1} \chi(x+b)\zeta_{p,E}\left(s,\frac{x+b}{p^{v}}\right)(-1)^{b}\\&=\zeta_{p,E}(\chi,s,x), \end{aligned}
\end{equation*}
using Theorem \ref{p-analytic-some-result23} (3) and (\ref{mid}).\end{proof}

\begin{corollary}\label{zeta-2}
Suppose that $x\in \Z.$ We have
$$\zeta_{p,E}(\chi,1,x)=\sum_{j=0}^{p^{v}-1}\chi(x+j)(-1)^{j}.$$
\end{corollary}
\begin{proof}
$$\zeta_{p,E}(\chi,1,x)=\sum_{j=0}^{p^{v}-1}\chi(x+j)\zeta_{p,E}\left(1,\frac{x+j}{p^{v}}\right)(-1)^{j}=\sum_{j=0}^{p^{v}-1}\chi(x+j)(-1)^{j}$$
using Corollary \ref{representation} and  Theorem \ref{zeta-1}.
\end{proof}

\begin{corollary}\label{C4.5}
Let $\chi$ be a character modulo $p^{v}$. Then for any $x\in\mathbb{Q}_{p}$ and any odd positive integer $N$ such that $p^{v}|N$ and $Nx\in\mathbb{Z}_{p}$ we have $$\sum_{j=0}^{N-1}\chi(Nx+j)\zeta_{p,E}\left(s,x+\frac{j}{N}\right)(-1)^{j}=\zeta_{p,E}(\chi,s,Nx).$$
In particular,
$$\sum_{\substack{j=0 \\ p\nmid(Nx+j)}}^{N-1}\zeta_{p,E}\left(s,x+\frac{j}{N}\right)(-1)^{j}=\zeta_{p,E}(s,Nx),$$
where we recall that we have defined $\zeta_{p,E}(s,x)=\zeta_{p,E}(\chi_{0},s,x)$ when $x\in\mathbb{Z}_{p}$.\end{corollary}
\begin{proof}
Follows from Corollary \ref{representation} applied to $M=N$ and $x$ is replaced by $Nx$.
\end{proof}

We now begin to study the properties for the functions $\zeta_{p,E}(\chi,s,x)$ for $x\in\mathbb{Z}_{p}$.

\begin{proposition}\label{der}
If $x\in\mathbb{Z}_{p}$ we have $$\frac{\partial\zeta_{p,E}}{\partial x}(\chi,s,x)=(1-s)\zeta_{p,E}(\chi\omega^{-1},s+1,x).$$\end{proposition}
\begin{proof} By  Corollary \ref{representation} and
Theorem \ref{p-analytic-some-result23} (4) we have \begin{equation*}
\begin{aligned}\frac{\partial\zeta_{p,E}}{\partial x}(\chi,s,x)&=(1-s)\sum_{j=0}^{p^{v}-1}\frac{\chi(x+j)}{p^{v}\omega_{v}\left(\frac{x+j}{p^{v}}\right)}
\zeta_{p,E}\left(s+1,\frac{x+j}{p^{v}}\right)(-1)^{j}\\&=(1-s)\sum_{j=0}^{p^{v}-1}\chi\omega^{-1}(x+j)
\zeta_{p,E}\left(s+1,\frac{x+j}{p^{v}}\right)(-1)^{j}\\&=(1-s)\zeta_{p,E}(\chi\omega^{-1},s+1,x).
\end{aligned}
\end{equation*}
This completes the proof.
\end{proof}

\begin{proposition}\label{analytic2}
For fixed $x\in\mathbb{Z}_{p}$ the function $\zeta_{p,E}(\chi,s,x)$ is a p-adic analytic function on $|s|_{p} < R_{p}$.
\end{proposition}
\begin{proof}
This is from Corollary \ref{representation} and Theorem \ref{p-analytic}.
\end{proof}

\begin{corollary}
We have $$\frac{\partial\zeta_{p,E}}{\partial x}(\chi\omega,0,x)=\sum_{j=0}^{p^{v}-1}\chi(x+j)(-1)^{j}.$$
\end{corollary}
\begin{proof}
By Proposition \ref{der}, Proposition \ref{analytic2} and Corollary \ref{zeta-2} we have
\begin{equation*}
\begin{aligned}\frac{\partial\zeta_{p,E}}{\partial x}(\chi\omega,0,x)&=\lim_{s\to 0}\frac{\partial\zeta_{p,E}}{\partial x}(\chi\omega,s,x)\\&=\lim_{s\to 0}(1-s)\zeta_{p,E}(\chi,s+1,x)\\&=\zeta_{p,E}(\chi,1,x)\\&=\sum_{j=0}^{p^{v}-1}\chi(x+j)(-1)^{j}.
\end{aligned}
\end{equation*}
This completes the proof.
\end{proof}

\begin{proposition}For $k>0$ we have $$\zeta_{p,E}(\chi\omega^{k},1-k,x) = p^{vk}\sum_{j=0}^{p^{v}-1}\chi(x+j)E_{k}\left(\frac{x+j}{p^{v}}\right)(-1)^{j}.$$
 \end{proposition}
\begin{proof} By Corollary \ref{representation} and Theorem \ref{p-E-zeta-val} (2) we have
 \begin{equation*}
\begin{aligned}
\zeta_{p,E}(\chi\omega^{k},1-k,x)&=\sum_{j=0}^{p^{v}-1}\chi\omega^{k}(x+j)\zeta_{p,E}
\left(1-k,\frac{x+j}{p^{v}}\right)(-1)^{j}\\
&=\sum_{j=0}^{p^{v}-1}\chi\omega^{k}(x+j)\omega_{v}\left(\frac{x+j}{p^{v}}\right)^{-k}E_{k}
\left(\frac{x+j}{p^{v}}\right)(-1)^{j}\\
&=p^{vk}\sum_{j=0}^{p^{v}-1}\chi(x+j)E_{k}\left(\frac{x+j}{p^{v}}\right)(-1)^{j}.
\end{aligned}
\end{equation*}
This completes the proof.
\end{proof}

\begin{theorem} \label{reflection}
Let $x\in\mathbb{Z}_{p}$ and $|s|_{p} < R_{p}$.
\begin{enumerate}
\item[\rm(1)] We have the functional equation $$\zeta_{p,E}(\chi,s,x+1)+\zeta_{p,E}(\chi,s,x)=2\chi(x)\langle x \rangle^{1-s}.$$
\item[\rm(2)] We have the reflection formula
$$\zeta_{p,E}(\chi,s,1-x)=\chi(-1)\zeta_{p,E}(\chi,s,x).$$
\item[\rm(3)] Set $$L_{p,E}(\chi,s)=\zeta_{p,E}(\chi,s,0).$$
If $\chi$ is an even character we have $L_{p,E}(\chi,s)=0$, and more generally if $n$ is a positive integer we have
$$\begin{aligned}
&\quad\zeta_{p,E}(\chi,s,n) \\
&=\chi(-1)\left(2\sum_{j=1}^{n-1}\frac{(-1)^{j+1}\chi(j-n)}{\langle j-n\rangle^{s-1}}
+(-1)^{n+1}L_{p,E}(\chi,s)\right).
\end{aligned}$$
\item[\rm(4)] If $p\nmid N$ and $N$ is odd we have the distribution formula
$$\sum_{i=0}^{N-1}\zeta_{p}\left(\chi,s,x+\frac{i}{N}\right)(-1)^{i}=\chi^{-1}(N)\zeta_{p,E}(\chi,s,Nx).$$
\end{enumerate}
\end{theorem}

\begin{remark}\label{re4.11}
If $\chi$ is a character modulo $p^{v}$, then we have
\begin{equation}
\begin{aligned}L_{p,E}(\chi,s)&=\zeta_{p,E}(\chi,s,0)\\&=\int_{\mathbb{Z}_{p}}\langle a \rangle^{1-s}\chi(a)d\mu_{-1}(a)\\&=\lim_{N\to\infty}\sum_{a=0}^{p^{N}-1}\langle a \rangle^{1-s}\chi(a)(-1)^{a}\\&= \lim_{N\to\infty}\sum_{\substack{a=0\\ p\nmid a }}^{p^{N}-1}\langle a \rangle^{1-s}\chi(a)(-1)^{a} ,
\end{aligned}
\end{equation}
so the definition of $L_{p,E}(\chi,s)$ in (3) is the same as the
first author's definition of $L_{p,E}(\chi,s)$ using
Kubota-Leopoldt's approach (see \cite[p. 6]{MSK2}).
\end{remark}

\begin{proof}
(1) By definition we have
\begin{equation*}
\begin{aligned}
&\quad\zeta_{p,E}(\chi,s,x+1)+\zeta_{p,E}(\chi,s,x) \\
&=\lim_{N\to\infty}\biggl(-\sum_{j=1}^{p^{N}}\chi(x+j)\langle x+j \rangle^{1-s}(-1)^{j}\\
&\quad+\sum_{j=0}^{p^{N}-1}\chi(x+j)\langle x+j \rangle^{1-s}(-1)^{j}\biggl)\\
&= \lim_{N\to\infty}(\chi(x)\langle x \rangle^{1-s}+\chi(x)\langle x+p^{N}\rangle^{1-s})\\
&=2\chi(x)\langle x \rangle^{1-s},
\end{aligned}
\end{equation*}
since $f(x)=\langle x\rangle^{1-s}$ is continuous for $|s|_{p}\leq R_{p}$ (see \cite[p.\,54]{Wa}).

(2) By Corollary \ref{representation}, Theorem \ref{p-analytic-some-result23} (2)
and setting $i=p^{v}-1-j$ we have
\begin{equation*}
\begin{aligned}&\quad\zeta_{p,E}(\chi,s,1-x) \\
&=\sum_{j=0}^{p^{v}-1}\chi(1-x+j)\zeta_{p,E}
\left(s,\frac{1-x+j}{p^{v}}\right)(-1)^{j}\\&=\sum_{i=0}^{p^{v}-1}\chi(p^{v}-i-x)\zeta_{p,E}\left(s,1-\frac{x+i}{p^{v}}\right)
(-1)^{i}\\&=\chi(-1)\sum_{i=0}^{p^{v}-1} \chi(x+i)\zeta_{p,E}\left(s,1-\frac{x+i}{p^{v}}\right)(-1)^{i}\\
&=\chi(-1)\sum_{i=0}^{p^{v}-1} \chi(x+i)\zeta_{p,E}\left(s,\frac{x+i}{p^{v}}\right)(-1)^{i}\\
&=\chi(-1)\zeta_{p,E}(\chi,s,x).
\end{aligned}
\end{equation*}

(3) From (1), we have
$$\zeta_{p,E}(\chi,s,1)=-\zeta_{p,E}(\chi,s,0)$$
for any character.
If $\chi(-1)=1$, we have  $\zeta_{p,E}(\chi,s,1)=\zeta_{p,E}(\chi,s,0)$ by (2),
thus
$$L_{p,E}(\chi,s)=\zeta_{p,E}(\chi,s,0)=0.$$
For any positive integer $n$, we have
\begin{equation*}
\begin{aligned}
&\quad\zeta_{p,E}(\chi,s,n)\\&=\chi(-1)\zeta_{p,E}(\chi,s,1-n)\\
&=\chi(-1)(2 \langle 1-n \rangle^{1-s}\chi(1-n)-\zeta_{p,E}(\chi,s,2-n))\\
&=\chi(-1)\left(
2\sum_{j=1}^{n-1}\frac{(-1)^{j+1}\chi(j-n)}{\langle j-n\rangle^{s-1}}
+(-1)^{n+1}L_{p,E}(\chi,s)\right)
\end{aligned}
\end{equation*}
by using (1)  and (2).

(4) By corollary \ref{representation} we have
\begin{equation*}
\begin{aligned}&\quad\sum_{i=0}^{N-1}\zeta_{p,E}\left(\chi,s,x+\frac{i}{N}\right)(-1)^{i}\\&=\sum_{i=0}^{N-1} \sum_{j=0}^{p^{v}-1}\chi\left(x+j+\frac{i}{N}\right)\zeta_{p,E}\left(s,\frac{x+j+i/N}{p^{v}}\right)(-1)^{j+i}.
\end{aligned}
\end{equation*}
Setting $a=Nj+i$ and using the fact that $p \nmid N$ and $N$ is an odd positive integer, we have
\begin{equation*}
\begin{aligned}
&\quad\sum_{i=0}^{N-1} \zeta_{p,E}\left(\chi,s,x+\frac{i}{N}\right)(-1)^{i}\\
&=\chi^{-1}(N)\sum_{a=0}^{Np^{v}-1} \chi(Nx+a)\zeta_{p,E}\left(s,\frac{Nx+a}{Np^{v}}\right)(-1)^{a}\\&=\chi^{-1}(N)
\zeta_{p,E}(\chi,s,Nx)
\end{aligned}
\end{equation*}
using Corollary \ref{representation}.
 \end{proof}

Next we prove the Raabe formula  of the  $p$-adic Hurwitz-type Euler zeta functions for $x\in\mathbb{Z}_{p}$.

\begin{theorem}
Let $\chi$ be a character modulo $p^{v}$. For $|s|_{p}<R_{p}$ and $x\in\mathbb{Z}_{p}$
we have $$\int_{\mathbb{Z}_{p}}\zeta_{p,E}(\chi,s,x+a)d\mu_{-1}(a)=2(1-x)\zeta_{p,E}(\chi,s,x)+2\zeta_{p,E}(\chi\omega,s-1,x).$$
\end{theorem}
\begin{proof}
From the definition and Corollary \ref{representation}, we have
\begin{equation*}
\begin{aligned}
&\quad\int_{\mathbb{Z}_{p}}\zeta_{p,E}(\chi,s,x+a)d\mu_{-1}(a)\\
&=\lim_{N\to\infty}\sum_{i=0}^{p^N-1}\zeta_{p,E}(\chi,s,x+i)(-1)^{i}\\
&=\lim_{N\to\infty}\sum_{i=0}^{p^N-1}\left(\sum_{j=0}^{p^N-1}
\chi(x+i+j)\zeta_{p,E}\left(\chi,s,\frac{x+i+j}{p^N}\right)(-1)^{j}\right)(-1)^{i}.
\end{aligned}
\end{equation*}
Setting $n\equiv i+j ~\textrm{mod}~ p^N$ in other words the unique  $n\equiv i+j ~\textrm{mod}~ p^N$ such that $0\leq n < p^N$.
We can have only $i+j=n$ or $i+j=n+p^N$ and the number of pairs $(i,j)$ such that $i+j=n$ is equal to $n+1$,
which the number of pairs such that
$i+j=n+p^N$ is equal to $p^N-(n+1)$.
Let $$S(N)=\sum_{i=0}^{p^N-1}\left(\sum_{j=0}^{p^N-1}\chi(x+i+j)\zeta_{p,E}\left(\chi,s,\frac{x+i+j}{p^N}\right)(-1)^{j
}\right)(-1)^{i}.$$
We have
$$\int_{\mathbb{Z}_{p}}\zeta_{p,E}(\chi,s,x+a)d\mu_{-1}(a)=\lim_{N\to\infty}S(N)$$ and
\begin{equation}\label{1}
\begin{aligned}
S(N)&=\sum_{n=0}^{p^N-1}(-1)^{n}\chi(x+n)\biggl((n+1) \zeta_{p,E}\left(s,\frac{x+n}{p^N}\right)
\\&\quad-(p^N-(n+1))\zeta_{p,E}\left(s,\frac{x+n}{p^N}+1\right)\biggl).
\end{aligned}
\end{equation}
By Theorem \ref{p-analytic-some-result23} (1), we have
\begin{equation}\label{2}
\begin{aligned}
\zeta_{p,E}\left(s,\frac{x+n}{p^N}+1\right)& + \zeta_{p,E}\left(s,\frac{x+n}{p^N}\right)=2\langle x+n \rangle^{1-s}.
\end{aligned}
\end{equation}
By (\ref{2}) we have
\begin{equation}\label{3}
\zeta_{p,E}\left(s,\frac{x+n}{p^N}+1\right)=2\langle x+n \rangle^{1-s}-\zeta_{p,E}\left(s,\frac{x+n}{p^N}\right).
\end{equation}
Substitute (\ref{3}) into (\ref{1}) we have
$$
\begin{aligned}
S(N)&=\sum_{n=0}^{p^N-1}(-1)^{n}\chi(x+n)\biggl((n+1)\zeta_{p,E}\left(s,\frac{x+n}{p^N}\right)\\
&\quad-(p^N-(n+1))\left(2\langle x+n \rangle^{1-s}-\zeta_{p,E}\left(s,\frac{x+n}{p^N}\right)\right)\biggl)\\
&=\sum_{n=0}^{p^N-1}(-1)^{n}\chi(x+n)(n+1)\zeta_{p,E}\left(s,\frac{x+n}{p^N}\right)\\
&\quad-\sum_{n=0}^{p^N-1}\chi(x+n)(p^N-(n+1))2 \langle x+n \rangle^{1-s}\\
&\quad+\sum_{n=0}^{p^N-1}(-1)^{n}\chi(x+n)(p^N-(n+1))\zeta_{p,E}\left(s,\frac{x+n}{p^N}\right)\\
&=\sum_{n=0}^{p^N-1}(-1)^{n}\chi(x+n)p^N\zeta_{p,E}\left(s,\frac{x+n}{p^N}\right)\\
&\quad-\sum_{n=0}^{p^N-1}(-1)^{n}\chi(x+n)(p^N-(n+1))2 \langle x+n \rangle^{1-s}.
\end{aligned}
$$
By Corollary \ref{representation} we have
$$
\quad S(N)=p^N\zeta_{p,E}(\chi,s,x)
-\sum_{n=0}^{p^N-1}(-1)^{n}\chi(x+n)(p^N-(n+1))2 \langle x+n \rangle^{1-s}.$$
From above equality we have
$$\begin{aligned}
\lim_{N\to\infty}S(N)=-2\lim_{N\to\infty}\sum_{n=0}^{p^N-1}(-1)^{n}\chi(x+n)(p^N-(n+1)) \langle x+n \rangle^{1-s}.
\end{aligned}$$
Since
$$
\begin{aligned}
&\quad\sum_{n=0}^{p^N-1}(-1)^{n}\chi(x+n)(p^N-(n+1)) \langle x+n \rangle^{1-s}\\
&=\sum_{n=0}^{p^N-1}(-1)^{n}\chi(x+n)(p^N+x-1-(x+n)) \langle x+n \rangle^{1-s}\\
&=p^N\sum_{n=0}^{p^N-1}(-1)^{n}\chi(x+n)\langle x+n \rangle^{1-s}\\
&\quad+\sum_{n=0}^{p^N-1}(-1)^{n}(x-1)\langle x+n \rangle^{1-s}\chi(x+n)\\
&\quad-\sum_{n=0}^{p^N-1}(-1)^{n}\chi(x+n)\omega_{v}(x+n)\langle x+n \rangle \langle x+n \rangle^{1-s}\\
&=p^N\sum_{n=0}^{p^N-1}(-1)^{n}\chi(x+n)\langle x+n \rangle^{1-s}\\
&\quad+\sum_{n=0}^{p^N-1}(-1)^{n}(x-1)\langle x+n \rangle^{1-s}\chi(x+n)\\
&\quad-\sum_{n=0}^{p^N-1}(-1)^{n}\chi\omega(x+n)\langle x+n \rangle^{1-(s-1)}
 \end{aligned}
$$
and
$$\lim_{N\to\infty}p^N\sum_{n=0}^{p^N-1}(-1)^{n}\chi(x+n)\langle x+n \rangle^{1-s}=0,$$
$$\lim_{N\to\infty}\sum_{n=0}^{p^N-1}(-1)^{n}(x-1)\langle x+n \rangle^{1-s}\chi(x+n)=(x-1)\zeta_{p,E}(\chi,s,x),$$
we have
$$
\begin{aligned}
-\lim_{N\to\infty}\sum_{n=0}^{p^N-1}(-1)^{n}\chi\omega(x+n)\langle x+n \rangle^{1-(s-1)}=-\zeta_{p,E}(\chi\omega,s-1,x),
\end{aligned}
$$
and
$$
\begin{aligned}
\int_{\mathbb{Z}_{p}}\zeta_{p,E}(\chi,s,x+a)d\mu_{-1}(a)&=\lim_{N\to\infty}S(N)\\
&=-2((x-1)\zeta_{p,E}(\chi,s,x)-\zeta_ {p,E}(\chi\omega,s-1,x))\\
&=2(1-x)\zeta_{p,E}(\chi,s,x)+2\zeta_{p,E}(\chi\omega,s-1,x).
\end{aligned}
$$
This completes the proof.
\end{proof}
Next we prove the following power series expansion in $x$ of $\zeta_{p,E}(\chi,s,x)$.
\begin{proposition}
Let $\chi$ be a character modulo $p^{v}$ for some $v\geq 1$. For $x\in p^{v}\mathbb{Z}_{p}$ we have the power series expansion
$$\zeta_{p,E}(\chi,s,x)=\sum_{k=0}^{\infty}\binom{1-s}{k}L_{p,E}(\chi\omega^{-k},s+k)x^{k},$$
where $L_{p,E}(\chi,s)=\zeta_{p,E}(\chi,s,0)$ is the $p$-adic Euler
$L$-function.
\end{proposition}
\begin{proof}
By definition, we have
\begin{equation*}
\begin{aligned}
\zeta_{p,E}(\chi,s,x)&=\int_{\mathbb{Z}_{p}}\chi(x+a)\langle x+a
\rangle^{1-s}
d\mu_{-1}(a)\\&=\int_{\mathbb{Z}_{p}^{\times}}\chi(a)\langle a
\rangle^{1-s}\langle 1+x/a \rangle^{1-s}
d\mu_{-1}(a)\\&=\int_{\mathbb{Z}_{p}^{\times}}\sum_{k=0}^{\infty}\binom{1-s}{k}\chi\omega^{-k}(a)\langle
a
\rangle^{1-s-k}x^{k}d\mu_{-1}(a)\\&=\sum_{k=0}^{\infty}\binom{1-s}{k}x^{k}\int_{\mathbb{Z}_{p}}\chi\omega^{-k}(a)\langle
a
\rangle^{1-s-k}d\mu_{-1}(a)\\&=\sum_{k=0}^{\infty}\binom{1-s}{k}\zeta_{p,E}(\chi\omega^{-k},s+k,0)x^{k}\\
&=\sum_{k=0}^{\infty}\binom{1-s}{k}L_{p,E}(\chi\omega^{-k},s+k)x^{k},
\end{aligned}
\end{equation*}
where $\mathbb{Z}_{p}^{\times}=\mathbb{Z}_{p}\setminus p\mathbb{Z}_{p}$.
This completes the proof.
\end{proof}

\section{$p$-adic Euler $L$-functions}
In section 5, we define the $p$-adic Euler $L$-functions by using
the $p$-adic Hurwitz-type Euler zeta functions. First we recall some
facts on the generalized Euler numbers which will be used in the
subsequent sections to study the properties for the $p$-adic Euler
$L$-functions.

\subsection{Generalized Euler numbers}

We consider the following generating function
\begin{equation}\label{def-E}
F(t,x)=\frac{2e^{xt}}{e^t+1}.
\end{equation}
Expand $F(t,x)$ into a power series of $t:$
\begin{equation}\label{def-E-pow}
F(t,x)=\sum_{n=0}^\infty E_{n}(x)\frac{t^n}{n!}
\end{equation}
(see \cite{CE}). The coefficients $E_n(x),n\geq0,$ are called Euler polynomials.

We generalize the above definition of $E_n(x)$ as follows: let
$\chi$ be a primitive Dirichlet character $\chi$ with an odd
conductor $f=f_\chi,$ and let
\begin{equation}\label{gen-E-numb}
F_\chi(t)=2\sum_{a=1}^{f}\frac{(-1)^a\chi(a)e^{at}}{e^{ft}+1},\quad
|t|<\frac\pi f
\end{equation}
and
\begin{equation}\label{gen-E-poly}
F_\chi(t,x)=F_\chi(t)e^{xt}=2\sum_{a=1}^{f}\frac{(-1)^a\chi(a)e^{(a+x)t}}{e^{ft}+1},\quad
|t|<\frac\pi f.
\end{equation}
(see e.g. \cite[p. 783]{TK}). Expanding theses into power series of
$t,$ let
\begin{equation}\label{gen-E-numb-d}
F_\chi(t)=\sum_{n=0}^{\infty}E_{n,\chi}\frac{t^n}{n!}
\end{equation}
and
\begin{equation}\label{gen-E-poly-d}
F_\chi(t,x)=\sum_{n=0}^{\infty}E_{n,\chi}(x)\frac{t^n}{n!}.
\end{equation}
Then $E_{n,\chi}(x)=\sum_{i=0}^n\binom niE_{i,\chi}x^{n-i}$ for
$n\geq0.$ The coefficients $E_{n,\chi}$ and $E_{n,\chi}(x)$ for
$n\geq0,$ are called generalized Euler numbers and polynomials,
respectively. It is clear from (\ref{gen-E-numb}) that
$F_\chi(-t)=-\chi(-1)F_\chi(t),$ if $\chi\neq\chi^0,$ the trivial
character. Hence we obtain the result as follows:

\begin{proposition}\label{E-0-pro}
If $\chi\neq\chi^0,$ the trivial character, then we have
$$(-1)^{n+1}E_{n,\chi}=\chi(-1)E_{n,\chi},\quad n\geq0.$$
In particular we obtain
$$E_{n,\chi}=0\quad\text{if }\chi\neq\chi^0,\; n\not\equiv\delta_\chi\pmod2,$$
where $\delta_\chi=0$ if $\chi(-1)=-1$ and $\delta_\chi=1$ if
$\chi(-1)=1.$
\end{proposition}

Let $N$ be the least common multiple of $p$ and $f.$ Then by
(\ref{gen-E-numb}), we have
\begin{equation}\label{rabbe}
F_{\chi}(t)=2\sum_{m=0}^\infty(-1)^m\chi(m)e^{mt}=2\sum_{a=1}^{N}(-1)^a\chi(a)\frac{(e^{\frac
aN})^{Nt}}{e^{Nt}+1}.
\end{equation}
Therefore, by (\ref{def-E}), (\ref{def-E-pow}), (\ref{gen-E-numb}),
(\ref{gen-E-numb-d}) and (\ref{rabbe}), we obtain the result as
follows:
\begin{proposition}\label{E-N-pro}
Let $N$ be the least common multiple of $p$ and $f$. Then
$$E_{n,\chi}=N^n\sum_{a=1}^{N}(-1)^a\chi(a)E_{n}\left(\frac{a}{N}\right).$$
\end{proposition}

From (\ref{gen-E-poly}), the generating function $F_\chi(t,x)$ is
given by
\begin{equation}\label{power-sum}
\begin{aligned}
F_\chi(t,x)&=2\sum_{a=0}^{f}(-1)^a\chi(a)\sum_{k=0}^\infty(-1)^ke^{(a+x+fk)t} \\
&=\sum_{n=0}^\infty\left(2\sum_{l=0}^\infty(-1)^l\chi(l)(l+x)^n\right)
\frac{t^n}{n!}.
\end{aligned}
\end{equation}
Comparing coefficients of ${t^n}/{n!}$ on both sides of
(\ref{gen-E-poly-d}) and (\ref{power-sum}) gives
\begin{equation}\label{power-sum-ell}
E_{n,\chi}(x)=2\sum_{l=0}^\infty(-1)^l\chi(l)(l+x)^n
\end{equation}
(see \cite[p.784]{TK}). In particular, if $n\geq0$ we have
\begin{equation}\label{power-sum-ell-N}
E_{n,\chi}(x+N)=2\sum_{l=0}^\infty(-1)^l\chi(l)(l+x+N)^n.
\end{equation}
By (\ref{power-sum-ell}) and (\ref{power-sum-ell-N}), we obtain the
results as follows:

\begin{proposition}\label{E-sum-ex}
Let $N$ be the least common multiple of $f$ and $p.$ Then
$$\sum_{r=0}^{N-1}(-1)^r\chi(r)(x+r)^n=\frac{E_{n,\chi}(x)-(-1)^NE_{n,\chi}(x+N)}{2}.$$
\end{proposition}

We  can also represent the generalized Euler numbers  using the
fermionic $p$-adic integral $I_{-1}(f)$ on $\mathbb Z_p$. Let $\chi$
be a character with conductor a power of $p$. Then by (\ref{int-Em})
and Proposition \ref{E-N-pro}, we see that
\begin{equation}\label{*}\begin{aligned}
\int_{\mathbb Z_p} \chi(x)x^n d\mu_{-1}(x)
&=\lim_{r\rightarrow\infty}\sum_{a=0}^{p^{r+1}-1}(-1)^a\chi(a)a^{n} \\
&=\lim_{r\rightarrow\infty}\sum_{k=0}^{p-1}
\sum_{a=0}^{p^r-1}(-1)^{pa+k}\chi(pa+k)(pa+k)^{n} \\
&=p^n\sum_{k=0}^{p-1}(-1)^{k}\chi(k) \int_{\mathbb Z_p}
\left(\frac{k}{p}+a\right)^{n}d\mu_{-1}(a) \\
&=p^n\sum_{k=0}^{p-1}(-1)^{k}\chi(k)E_n\left(\frac kp\right) \\
&=E_{n,\chi}
\end{aligned}\end{equation}
is equivalent to
\begin{equation}
I_{-1}(\chi(x)x^n)=\int_{\Z}\chi(x)x^n d\mu_{-1}(x)=E_{n,\chi}.
\end{equation}
Similarly, by (\ref{*}) we have
\begin{equation}
I_{-1}(\chi(y)(x+y)^n)=\int_{\Z}\chi(y)(x+y)^n
d\mu_{-1}(y)=E_{n,\chi}(x).
\end{equation}
Therefore we have the results as follows:

\begin{proposition}\label{E-np-pro}
Let $\chi$ be a primitive Dirichlet character with conductor a power
of $p$. Then for $n\geq 0$
\begin{itemize}
\item[\rm(1)] $I_{-1}(\chi(x)x^n)=\int_{\Z}\chi(x)x^n d\mu_{-1}(x)=E_{n,\chi}.$
\item[\rm(2)] $I_{-1}(\chi(y)(x+y)^n)=\int_{\Z}\chi(y)(x+y)^n d\mu_{-1}(y)=E_{n,\chi}(x).$
\end{itemize}
\end{proposition}

\subsection{$p$-adic Euler $L$-functions}
In this subsection, we define the $p$-adic Euler $L$-functions and
compute its value at negative integers.

\begin{definition}\label{11.3.4}
Let $\chi$ be a primitive character of conductor $f$ and $m$ be the
least common multiple of $f$ and $p$. For $s\in\C$ such that
$|s|<R_p=p^{(p-2)/(p-1)},$ we have
$$L_{p,E}(\chi,s)=\langle m\rangle^{1-s}\sum_{a=0}^{f-1}\chi(a)\zeta_{p,E}\left(s,\frac af\right)(-1)^a,$$
where the $p$-adic Hurwitz-type Euler zeta function
$\zeta_{p,E}\left(s, x\right)$ is defined in
Definition \ref{p-E-zeta} and \ref{de2} for $x\in\mathbb{C}_{p}$. We
define $L_{p,E}(\chi,1)=\lim\limits_{s\to1}L_{p,E}(\chi,s),$ when
the limit exist. In particular if $\chi$ is the trivial character
$\chi_0,$ we set $\zeta_{p,E}(s)=L_{p,E}(\chi_0,s)=\zeta_{p,E}(s,0)$
(see Theorem \ref{reflection} (3)).
\end{definition}

 \begin{definition}[{see \cite[p. 302]{Co2}}]\label{11.3.6}
\begin{enumerate}
\item[(1)] Let $m\in\BZ_{>0}.$ We define $\chi_{0,m}$ to be the trivial character modulo 1 when $p\nmid m,$
and to be the trivial character modulo $p$ when $p\mid m.$ In other
words, $\chi_{0,m}(a)=1$ when $p\nmid a$ or when $p\mid a$ but
$p\nmid m$ and $\chi_{0,m}(a)=0$ when $p\mid a$ and $p\mid m.$
\item[(2)] If $I\subset\BZ,$ we set
$$\psum_{a\in I}g(a)=\sum_{\substack{a\in I \\ p\nmid a}}g(a).$$
In particular, if $p\mid m$ we have
$$\psum_{0\leq a<m}g(a)=\sum_{a=0}^{m-1}\chi_{0,m}(a)g(a).$$
\end{enumerate}
\end{definition}

\begin{lemma}[{see \cite[p. 302]{Co2}}]\label{11.3.7}
Let $\chi$ be a nontrivial primitive character of conductor $f$ and
$m$ be a common multiple of $f$ and $p.$ Then
$$\psum_{0\leq a<m}\chi(a)=0.$$
\end{lemma}

\begin{proposition}\label{11.3.8}
Let $\chi$ be a primitive character of an odd conductor $f$ and $m$
be the least common multiple of $f$ and $p$. Let $s\in\C$ be such
that $|s|< R_p.$
\begin{enumerate}
\item[\rm(1)] We have
$$L_{p,E}(\chi,s)=\langle m\rangle^{1-s}\sum_{a=0}^{m-1}\chi_{0,m}(a)\chi(a)\zeta_{p,E}\left(s,\frac am\right)(-1)^a.$$

\item[\rm(2)] If, in addition, $p\mid m,$ we have
$$L_{p,E}(\chi,s)=\langle m\rangle^{1-s}\psum_{0\leq a<m}\chi(a)(-1)^a\left\langle a \right\rangle^{1-s}
\sum_{i=0}^\infty\binom{1-s}i\frac{m^i}{a^i}E_i(0).$$

\item[\rm(3)] If $\chi\neq\chi_0,$ then
$$L_{p,E}(\chi,1)=\psum_{0\leq a<m}\chi(a)(-1)^a.$$
\end{enumerate}
\end{proposition}
\begin{proof}
(1) Writing $a=kf+r,$ we have
\begin{equation}\label{pf-1}
\begin{aligned}
&\quad\langle m\rangle^{1-s}\sum_{a=0}^{m-1}\chi_{0,m}(a)\chi(a)\zeta_{p,E}\left(s,\frac am\right)(-1)^a \\
&=\langle
m\rangle^{1-s}\sum_{r=0}^{f-1}\chi(r)(-1)^r\sum_{k=0}^{\frac
mf-1}\chi_{0,m}(kf+r) \zeta_{p,E}\left(s,\frac{kf+r}m\right)(-1)^k.
\end{aligned}
\end{equation}

Case (I): If $p\nmid m,$ then $\chi_{0,m}$ to be a trivial character
modulo 1 by Definition \ref{11.3.6}. Using Theorem \ref{p-analytic-some-result23} (3), we have
$$\begin{aligned}
\sum_{k=0}^{\frac
mf-1}\chi_{0,m}(kf+r)\zeta_{p,E}\left(s,\frac{kf+r}m\right)(-1)^k
&=\sum_{k=0}^{\frac mf-1}\zeta_{p,E}\left(s,\frac{r}m+\frac{k}{\frac mf}\right)(-1)^k \\
&=\zeta_{p,E}\left(s,\frac rf\right).
\end{aligned}$$

Case (II): If $p\mid m,$ then $\chi_{0,m}$ to be a trivial character
modulo $p$ by Definition \ref{11.3.6}. Thus we have
$$\sum_{k=0}^{\frac
mf-1}\chi_{0,m}(kf+r)\zeta_{p,E}\!\left(s,\frac{kf+r}m\right)(-1)^k
=\!\sum_{\substack{k=0 \\ p\nmid(kf+r)}}^{\frac
mf-1}\zeta_{p,E}\!\left(s,\frac{r}m+\frac{k}{\frac mf}\right)(-1)^k.$$
Using Corollary \ref{C4.5}, we obtain the following:

(a) Suppose $p\nmid f.$ Then $r/f\in\Z$ and
$$p\nmid(kf+r)\quad\text{if and only if}\quad p\nmid(r/f +k).$$
Thus
\begin{equation}\label{pf-2}
\begin{aligned}
\sum_{\substack{k=0 \\ p\nmid(kf+r)}}^{\frac
mf-1}\zeta_{p,E}\left(s,\frac{r}m+\frac{k}{\frac mf}\right)(-1)^k
=\zeta_{p,E}\left(s,\frac rf\right).
\end{aligned}
\end{equation}

(b) If $p\mid f,$ then $\chi_{0,m}(kf+r)=\chi_{0,m}(r).$ We have
$$\begin{aligned}
&\quad\sum_{k=0}^{\frac mf-1}\chi_{0,m}(kf+r)\zeta_{p,E}\left(s,\frac{kf+r}m\right)(-1)^k \\
&=\chi_{0,m}(r)\sum_{k=0}^{\frac mf-1}\zeta_{p,E}\left(s,\frac
rm+\frac k{\frac mf}\right)(-1)^k.
\end{aligned}$$
If $p\mid r,$ then
\begin{equation}\label{pf-3}
\sum_{k=0}^{\frac mf-1}\chi_{0,m}(kf+r)\zeta_{p,E}\left(s,\frac
rm+\frac k{\frac mf}\right)(-1)^k=0.
\end{equation}
If $p\nmid r,$ then $r/f\in C\Z.$ Thus by Theorem \ref{p-analytic-some-result23} (3) and the definition of
$\chi_{0,m}(r)$, we have
\begin{equation}\label{pf-4}
\begin{aligned}
\chi_{0,m}(r)\sum_{k=0}^{\frac mf-1}\zeta_{p,E}\left(s,\frac
rm+\frac k{\frac mf}\right)(-1)^k
&=\chi_{0,m}(r)\zeta_{p,E}\left(s,\frac rf\right) \\
&=\zeta_{p,E}\left(s,\frac rf\right).
\end{aligned}
\end{equation}
Thus if $p\nmid f,$ substitute (\ref{pf-2}) to (\ref{pf-1}), we have
$$\begin{aligned}
&\quad\langle m\rangle^{1-s}\sum_{a=0}^{m-1}\chi_{0,m}(a)\chi(a)\zeta_{p,E}\left(s,\frac am\right)(-1)^a \\
&=\langle
m\rangle^{1-s}\sum_{r=0}^{f-1}\chi(r)(-1)^r\sum_{k=0}^{\frac
mf-1}\chi_{0,m}(kf+r)
\zeta_{p,E}\left(s,\frac rm +\frac k{\frac mf}\right)(-1)^k \\
&=\langle m\rangle^{1-s}\sum_{r=0}^{f-1}\chi(r)(-1)^r\zeta_{p,E}\left(s,\frac rf\right) \\
&=L_{p,E}(\chi,s).
\end{aligned}$$
If $p\mid f,$ substitute (\ref{pf-3}) and (\ref{pf-4}) to
(\ref{pf-1}), we have
$$\begin{aligned}
&\quad\langle m\rangle^{1-s}\sum_{a=0}^{m-1}\chi_{0,m}(a)\chi(a)\zeta_{p,E}\left(s,\frac am\right)(-1)^a \\
&=\langle
m\rangle^{1-s}\sum_{r=0}^{f-1}\chi(r)(-1)^r\sum_{k=0}^{\frac
mf-1}\chi_{0,m}(kf+r)
\zeta_{p,E}\left(s,\frac rm +\frac k{\frac mf}\right)(-1)^k \\
&=\langle m\rangle^{1-s}\sum_{\substack{r=0 \\ p\nmid
r}}^{f-1}\chi(r)(-1)^r\sum_{k=0}^{\frac mf-1}\chi_{0,m}(kf+r)
\zeta_{p,E}\left(s,\frac rm +\frac k{\frac mf}\right)(-1)^k \\
&=\langle m\rangle^{1-s}\sum_{\substack{r=0 \\ p\nmid r}}^{f-1}\chi(r)(-1)^r\zeta_{p,E}\left(s,\frac rf\right) \\
&=L_{p,E}(\chi,s),
\end{aligned}$$
since $p\mid f$ and $p\mid r,$ we have $(f,r)\neq1,$ so $\chi(r)=0,$
thus we get the last equality.

(2) If $p\mid m,$ from (1), we have
$$\begin{aligned}
L_{p,E}(\chi,s)=\langle
m\rangle^{1-s}\sum_{a=0}^{m-1}\chi_{0,m}(a)\chi(a)\zeta_{p,E}\left(s,\frac
am\right)(-1)^a.
\end{aligned}$$
Since $p\mid m,$ by Definition \ref{11.3.6}, we have
\begin{equation}\label{pf-5}
\begin{aligned}
L_{p,E}(\chi,s)&=\langle m\rangle^{1-s}\sum_{a=0}^{m-1}\chi_{0,m}(a)\chi(a)\zeta_{p,E}\left(s,\frac am\right)(-1)^a \\
&=\langle m\rangle^{1-s}\psum_{0\leq
a<m}\chi(a)\zeta_{p,E}\left(s,\frac am\right)(-1)^a.
\end{aligned}
\end{equation}
Also since in the above equality, $p\nmid a$ and $p\mid m,$ we
obtain
$$\frac am\in C\Z.$$
By Theorem \ref{p-analytic}, we have
\begin{equation}\label{pf-6}
\zeta_{p,E}\left(s,\frac am\right) =\left\langle \frac
am\right\rangle^{1-s}
\sum_{i=0}^\infty\binom{1-s}iE_i(0)\frac{m^i}{a^i},
\end{equation}
where $E_i$ are Euler numbers. Substitute (\ref{pf-6}) to
(\ref{pf-5}), we have
$$L_{p,E}(\chi,s)=\psum_{0\leq a<m}\chi(a)(-1)^a \langle a \rangle^{1-s}
\sum_{i=0}^\infty\binom{1-s}i\frac{m^i}{a^i}E_i(0).$$

(3) Since $p\mid m,$ by (2), we have
$$L_{p,E}(\chi,1)=\psum_{0\leq a<m}\chi(a)(-1)^a$$
and the result is established.
\end{proof}

Next we compute the values of $p$-adic Euler $L$-functions at
negative integers.
\begin{proposition}\label{11.3.9}
Keep the above assumptions.
\begin{enumerate}
\item[\rm(1)]
The function $L_{p,E}(\chi,s)$ is a $p$-adic analytic function for
$|s|<R_p.$
\item[\rm(2)]
For $k\in\BZ_{\geq1},$ we have
$$
L_{p,E}(\chi,1-k)=(1-p^k\chi_k(p))E_{k,\chi_k},$$ where
$\chi_k=\chi\omega^{-k}$ and $\chi\not=\chi_{0}$, the trivial
character.
\item[\rm(3)]
If $\chi$ is an even character the function $L_{p,E}(\chi,s)$ is
identically equal to zero
\end{enumerate}
\end{proposition}
\begin{remark}\label{similar}
 By comparing Proposition \ref{11.3.9} (2) with equalities (3.2) and (3.3) in \cite[p. 6]{MSK2}, from Lemma 1 in \cite[p. 19]{Iw} we conclude that for Dirichlet characters with odd conductor the definition of $p$-adic Euler $L$-functions in this paper
is equivalent to the first author's previous definition in
\cite{MSK2} following Kubata-Leopoldt's approach.
\end{remark}
\begin{proof}
(1) By definition of $L_{p,E}(\chi,s),$ Theorem \ref{p-analytic} and
Proposition \ref{analytic2}, the result follows.

(2) Let $m$ be the least common multiple of $p$ and $f$. Then by
Proposition \ref{11.3.8} (1), we have
$$\begin{aligned}
L_{p,E}(\chi,1-k)=\langle
m\rangle^{k}\sum_{a=0}^{m-1}\chi_{0,m}(a)\chi(a)\zeta_{p,E}\left(1-k,\frac
am\right)(-1)^a.
\end{aligned}$$
Since $p\mid m,$ by Definition \ref{11.3.6} (2), we have
\begin{equation}\label{pf-7}
\begin{aligned}
L_{p,E}(\chi,1-k)=\langle m\rangle^{k}\psum_{0\leq
a<m}\chi(a)\zeta_{p,E}\left(1-k,\frac am\right)(-1)^a.
\end{aligned}
\end{equation}
Also since $p\nmid a$ and $p\mid m,$ we have $a/m\in C\Z.$ Using
Theorem \ref{p-E-zeta-val} (2) we have
\begin{equation}\label{pf-8}
\begin{aligned}
\zeta_{p,E}\left(1-k,\frac am\right)=\frac{1}{\omega_v^k\left(\frac
am\right)}E_k\left(\frac am\right).
\end{aligned}
\end{equation}
where $E_k(x)$ are the Euler polynomials. Substitute (\ref{pf-8}) to
(\ref{pf-7}), we have
$$\begin{aligned}
L_{p,E}(\chi,1-k)&=\langle m\rangle^{k}\psum_{0\leq
a<m}\chi(a)(-1)^a
\frac{1}{\omega_v^k\left(\frac am\right)}E_k\left(\frac am\right)\\
&=\langle m\rangle^{k}\omega_v^k(m)\psum_{0\leq
a<m}(-1)^a\chi\omega_v^{-k}(a) E_k\left(\frac am\right).
\end{aligned}$$
Since $p\nmid a,$ we have
$$\omega_v(a)=p^{v(a)}\omega\left(\frac{a}{p^{v(a)}}\right)=\omega(a).$$
Thus if $\chi\not=\chi_{0}$, by Proposition \ref{E-N-pro}, we have
$$\begin{aligned}
L_{p,E}(\chi,1-k) &=\langle m\rangle^{k}\omega_v^k(m)\psum_{0\leq
a<m}(-1)^a\chi\omega^{-k}(a)
E_k\left(\frac am\right) \\
&=\langle m\rangle^{k}\omega_v^k(m)\biggl(\sum_{a=0}^{m-1}(-1)^a\chi\omega^{-k}(a)E_k\left(\frac am\right) \\
&\quad-\sum_{\substack{a=0 \\ p\mid a}}^{m-1}(-1)^a\chi\omega_v^{-k}(a)E_k\left(\frac am\right)\biggl) \\
&=\biggl(m^k\sum_{a=0}^{m-1}(-1)^a\chi\omega^{-k}(a)E_k\left(\frac am\right) \\
&\quad-p^k\chi\omega^{-k}(p)
\left(\frac mp\right)^k\sum_{a=0}^{\frac mp-1}(-1)^a\chi\omega^{-k}(a)E_k\biggl(\frac {a}{\frac mp}\biggl)\biggl) \\
&=(1-p^k\chi_k(p))E_{k,\chi_k},
\end{aligned}$$
where $\chi_k=\chi\omega^{-k}.$

(3) From Definition \ref{11.3.4}, we have
$$L_{p,E}(\chi,s)=\langle m\rangle^{1-s}\sum_{a=0}^{f-1}\chi(a)\zeta_{p,E}\left(s,\frac af\right)(-1)^a.$$
Let $b=f-1-a,$ and let $\chi$ be an even character. Then by
Theorem \ref{p-analytic-some-result23} (2) and Theorem \ref{reflection}, we
have
$$\begin{aligned}
L_{p,E}(\chi,s)&=\langle m\rangle^{1-s}\sum_{b=0}^{f-1}\chi(f-1-b)\zeta_{p,E}\left(s,\frac{f-1-b}f\right)(-1)^{f-1-b} \\
&=\langle m\rangle^{1-s}\sum_{b=0}^{f-1}\chi(1+b)\zeta_{p,E}\left(s,\frac{1+b}f\right)(-1)^{b} \\
&=\langle
m\rangle^{1-s}\sum_{b'=1}^{f}\chi(b')\zeta_{p,E}\left(s,\frac{b'}f\right)(-1)^{b'}(-1),
\end{aligned}$$
so $L_{p,E}(\chi,s)=-L_{p,E}(\chi,s).$ Therefore $L_{p,E}(\chi,s)=0$
if $\chi$ is an even character.
\end{proof}

\subsection{$p$-adic Euler L-function at positive integers} In this
subsection we study the behavior of $p$-adic Euler $L$-functions at
positive integers following Cohen's approach in Section 11.3.3
of \cite{Co2}. We show that most of the results in Section 11.3.3 of
Cohen's book are also established if we replace the generalized
Bernoulli numbers with the generalized Euler numbers.
\begin{proposition} \label{11.3.10}
Let $f$ be an odd integer and $\chi$ be a primitive character modulo
$f$ and $m$ be the least common multiple of $f$ and $p$.
\begin{itemize}
\item[\rm(1)] For $k\in\mathbb{Z}\backslash\{0\}$, we have
$$L_{p,E}(\chi,k+1)=\lim_{N\to\infty}\psum_{0\leq n <  mp^{N}}\chi\omega^{k}(n)\frac{(-1)^{n}}{n^{k}}.$$
\item[\rm(2)] We have
$$L_{p,E}(\chi,1)= E_{0,\chi}, $$
where $E_{0,\chi}$ defined in (\ref{gen-E-numb-d}).
\end{itemize}
\end{proposition}
\begin{proof}
(1) By Theorem \ref{p-E-zeta-val} (1), for $p\nmid a$, we have
\begin{equation}\label{pf-9}
\begin{aligned}
\zeta_{p,E}\left(k+1,\frac{a}{m}\right)&=\omega_{v}^{k}\left(\frac{a}{m}\right)\int_{\mathbb{Z}_{p}}
\frac{1}{\left(\frac{a}{m}+j\right)^{k}}d\mu_{-1}(j)\\
&=\omega_{v}^{k}\left(\frac{a}{m}\right)\lim_{N\to\infty}\sum_{j=0}^{p^{N}-1}\frac{(-1)^{j}}{\left(\frac{a}{m}+j\right)^{k}}\\
&=\omega_{v}^{k}\left(\frac{a}{m}\right)m^{k}\lim_{N\to\infty}\sum_{j=0}^{p^{N}-1}\frac{(-1)^{j}}{(a+mj)^{k}}\\
&=\omega_{v}^{k}(a)\frac{m^{k}}{\omega_{v}^{k}(m)}\lim_{N\to\infty}\sum_{j=0}^{p^{N}-1}\frac{(-1)^{j}}{(a+mj)^{k}}.
\end{aligned}
\end{equation}
So that, by Definition \ref{11.3.6} (2) and Proposition \ref{11.3.8}
(1), we have
\begin{equation} \label{pf-10}
\begin{aligned}
L_{p,E}(\chi,k+1)&=\langle m \rangle^{-k}\sum_{a=0}^{m-1}\chi_{0,m}(a)\chi(a)\zeta_{p,E}\left(k+1,\frac{a}{m}\right)(-1)^{a}\\
&=\langle m \rangle^{-k}\psum_{0\leq a<
m}\chi(a)\zeta_{p,E}\left(k+1,\frac{a}{m}\right)(-1)^{a}.
\end{aligned}
\end{equation}
Thus if we set
$$n=mj+a, ~\textrm{where}~ 0\leq j\leq p^{N}-1, ~0\leq a\leq m-1,~ p\nmid a.$$
Then $$0\leq n\leq mp^{N}-1,~\textrm{and}~ p\nmid m.$$
Substitute (\ref{pf-9}) to (\ref{pf-10}), since $f$ is an odd
integer and $m$ is the least common multiple of $f$ and $p$, we have
\begin{equation}\label{pf-11}
\begin{aligned}
L_{p,E}(\chi,k+1)&=\psum_{0\leq a < m}\chi(a)\omega_{v}^{k}(a)(-1)^{a}\lim_{N\to\infty}\sum_{j=0}^{p^{N}-1}\frac{(-1)^{j}}{(a+mj)^{k}}\\
&=\psum_{0\leq a < m}\chi(a)\omega_{v}^{k}(a)\lim_{N\to\infty}\sum_{j=0}^{p^{N}-1}\frac{(-1)^{j+a}}{(a+mj)^{k}}\\
&=\psum_{0\leq a < m}\chi\omega^{k}(a)\lim_{N\to\infty}\sum_{j=0}^{p^{N}-1}\frac{(-1)^{ma+j}}{(ma+j)^{k}}\\
&=\lim_{N\to\infty}\psum_{0\leq n <
mp^{N}-1}\chi\omega^{k}(n)\frac{(-1)^{n}}{n^{k}}.
\end{aligned}
\end{equation}

(2) By Definition of $L_{p,E}(\chi,s)$, Theorem \ref{zeta-1} and
Corollary \ref{zeta-2}, we have
$$\begin{aligned}
L_{p,E}(\chi,k+1)&=\sum_{a=0}^{f-1}\chi(a)\zeta_{p,E}\left(1,\frac{a}{f}\right)(-1)^{a}\\
&=\sum_{a=0}^{f-1}\chi(a)(-1)^{a}\\
&=E_{0,\chi}.
\end{aligned}
$$
This completes the proof.
\end{proof}

\begin{corollary}\label{11.3.11}
Let $k\in\mathbb{Z}\backslash\{0\}$. If $\chi$ is a primitive
character modulo a power of $p$, then we have
$$L_{p,E}(\chi,k+1)=\int_{\mathbb{Z}_{p}^{\times}}\frac{\chi\omega^{k}(x)}{x^{k}}d\mu_{-1}(x).$$
In particular,
$$L_{p,E}(\omega^{-k},k+1)=\int_{\mathbb{Z}_{p}^{\times}}\frac{1}{x^{k}}d\mu_{-1}(x).$$
\end{corollary}
\begin{proof}
If $f$ is a power of $p$, that is $f=p^{v}$, then we have
$$\begin{aligned}
L_{p,E}(\chi,s)&=\sum_{a=0}^{p^{v}-1}\chi(a)\zeta_{p,E}\left(s,\frac{a}{p^{v}}\right)(-1)^{a}\\
&=\zeta_{p,E}(\chi,s,0)\\
&=\ell_{p,E}(\chi,s)
\end{aligned}
$$
by Corollary \ref{representation} and Remark \ref{re4.11}. Thus
$$\begin{aligned}
L_{p,E}(\chi,s)&=\ell_{p,E}(x,k+1)\\
&=\zeta(\chi,s,0)\\
&=\int_{\mathbb{Z}_{p}}\chi(x)\langle x \rangle^{-k}d\mu_{-1}(x)\\
&=\int_{\mathbb{Z}_{p}^{\times}}\chi(x)\langle x \rangle^{-k}d\mu_{-1}(x)\\
&=\int_{\mathbb{Z}_{p}^{\times}}\chi\omega_{k}(x)x^{-k}d\mu_{-1}(x)\\
&=\int_{\mathbb{Z}_{p}^{\times}}\frac{\chi\omega_{k}(x)}{x^{k}}d\mu_{-1}(x).
\end{aligned}
$$
This completes the proof.
\end{proof}

\begin{proposition}\label{11.3.12}
Let $\chi$ be a primitive character modulo $f$ and $\Phi$ be the
Euler-phi function. Then for all $k\in\mathbb{Z}$, we have
$$L_{p,E}(\chi,k+1)=\lim_{r\to\infty}E_{\Phi(p^{r})-k,\chi\omega^{k}}.$$
In particular, $$\lim_{r\to\infty}E_{\Phi(p^{r})-k}(0)=
L_{p,E}(\omega^{-k},k+1).$$
\end{proposition}
\begin{proof}Denote $\chi_{k}=\chi\omega^{-k}$. Since $L_{p,E}(\chi,s)$ is a continuous function of $s$, for all $k\in\mathbb{Z}$, we have
$$\begin{aligned}
L_{p,E}(\chi,k+1)&=\lim_{r\to\infty}L_{p,E}(\chi,k+1-\Phi(p^{r}))\\
&=\lim_{r\to\infty}L_{p,E}(\chi,1-(\Phi(p^{r})-k)\\
&=\lim_{r\to\infty}(1-
p^{\Phi(p^{r})-k}\chi_{\Phi(p^{r})-k}(p))E_{\Phi(p^{r})-k,\chi\omega^{k-\Phi(p^{r})}}
\end{aligned}
$$
using Proposition \ref{11.3.9}. Since
$$\omega^{\Phi(p^{r})}=\omega^{p^{r-1}(p-1)}=1\quad\text{and}
\quad\chi_{\Phi(p^{r})-k}=\chi\omega^{k-\Phi(p^{r})}=\chi\omega^{k},$$
thus
$$L_{p,E}(\chi,k+1)=\lim_{r\to\infty}E_{\Phi(p^{r})-k,\chi\omega^{k}}.$$
This completes the proof of the Theorem.
\end{proof}

\begin{definition}\label{11.3.13}
For $k\in\mathbb{Z}$, we define the $p$-adic $\chi$-Euler numbers by
$$
E_{k,p,\chi}=\lim_{r\to\infty}E_{\Phi(p^{r})+k,\chi}=L_{p,E}(\chi\omega^{k},1-k).
$$
\end{definition}

\begin{proposition}\label{11.3.14}
Assume that the conductor of $\chi$ is a power of $p$. Then for
$k\in\mathbb{Z}$, we have
$$E_{k,p,\chi}=\lim_{r\to\infty}\psum_{0\leq n < p^{r}}\chi(n)n^{k}(-1)^{n}
=\int_{\mathbb{Z}_{p}}\chi(x)x^{k}d\mu_{-1}(x).$$
\end{proposition}
\begin{proof} We have
$$\begin{aligned}
E_{k,p,\chi}&=L_{p,E}(\chi\omega^{k},1-k)\\
&=\int_{\mathbb{Z}_{p}^{\times}}\chi(x)x^{k}d\mu_{-1}(x)\\&=\lim_{r\to\infty}\psum_{0\leq
n < p^{r}}\chi(n)n^{k}(-1)^{n}.
\end{aligned}
$$
using Corollary \ref{11.3.11}. Thus
$$
E_{k,p,\chi}=\lim_{r\to\infty}\psum_{0\leq n <
p^{r}}\chi(n)n^{k}(-1)^{n}
=\int_{\mathbb{Z}_{p}}\chi(x)x^{k}d\mu_{-1}(x).
$$
This completes the proof.
\end{proof}

\begin{proposition}\label{11.3.15}
\begin{itemize}
\item[{\rm (1)}] If $\chi(-1)=(-1)^{k}$, then we have
$$E_{k,p,\chi}=0.$$
\item[{\rm (2)}] If $k\geq 1$, then we have
$$E_{k,p,\chi}=(1-p^{k}\chi(p))E_{k,\chi}.$$
\item[{\rm (3)}] Let $m$ be the least common multiple of $f$ and $p$, and set
$$H_{n}(x)=\psum_{0\leq a < m}\frac{\chi(a)}{a^{n}}(-1)^{a}.$$
If $k\geq 1$ and $\chi(-1)=(-1)^{k-1}$, we have
$$E_{-k,p,\chi}=\sum_{i=0}^{m}(-1)^{i}\binom{k+i-1}{k-1} m^{i}E_{i}(0)H_{k+i}(x).$$
\item[{\rm (4)}] For all $k$, we have $v_{p}(E_{k,p,\chi})\geq 0$.
\end{itemize}
\end{proposition}
\begin{proof}
(1) If $\chi(-1)=1$ and $k\equiv 0 ~(\textrm{mod}~ 2)$, then $\Phi(p^{r})+ k\equiv 0 ~(\textrm{mod}~ 2), $ we have $E_{\Phi(p^{r}) + k,\chi}=0$ by Proposition \ref{E-0-pro}, thus $$E_{k,p,\chi}=\lim_{r\to\infty}E_{\Phi(p^{r})+ k,\chi}=0.$$
If $\chi(-1)=-1$ and $k\equiv 1~ (\textrm{mod}~ 2)$, then
$\Phi(p^{r})+k\equiv 1 ~(\textrm{mod}~ 2), $ we have
$E_{\Phi(p^{r})+k,\chi}=0$ by Proposition \ref{E-0-pro}, thus
$$E_{k,p,\chi}=\lim_{r\to\infty}E_{\Phi(p^{r})+k,\chi}=0.$$

(2) By Proposition \ref{11.3.9} (2) and Definition \ref{11.3.13}, we have
$$
E_{k,p,\chi}=L_{p,E}(\chi\omega^{k},1-k)=(1-p^{k}\chi(p))E_{k,\chi}.
$$

(3) By Proposition  \ref{11.3.8}, we have
$$\begin{aligned}
L_{p,E}(\chi\omega^{-k},k+1)&=\langle m \rangle^{-k}\psum_{0\leq a < m}\chi\omega^{-k}(a)(-1)^{a}\left\langle\frac{a}{m}\right\rangle^{-k}\sum_{i=0}^{\infty}\binom{-k}i\frac{m^{i}}{a^{i}}E_{i}(0)\\
&=\psum_{0\leq a < M}\chi_{k}(a)(-1)^{a}\langle a \rangle^{-k}\sum_{i=0}^{\infty}\binom{-k}i\frac{m^{i}}{a^{i}}E_{i}(0)\\
&=\sum_{i=0}^{\infty}\binom{-k}i m^{i}E_{i}(0)\psum_{0\leq a < m}\chi_{k}(a)(-1)^{a}\frac{\langle a\rangle^{-k}}{a^{i}}\\
&=\sum_{i=0}^{\infty}\binom{-k}i m^{i}E_{i}(0)\psum_{0\leq a < m}\chi\omega^{-k}(a)(-1)^{a}\frac{\omega^{k}(a)}{a^{i+k}}\\
&=\sum_{i=0}^{\infty}\binom{-k}i m^{i}E_{i}(0)\psum_{0\leq a <
m}\chi(a)(-1)^{a}\frac{1}{a^{i+k}}.
\end{aligned}$$
 By Definition \ref{11.3.13} (1), we have
$$\begin{aligned}
E_{-k,p,\chi}&=L_{p,E}(\chi\omega^{-k},k+1)\\
&=\sum_{i=0}^{\infty}\binom{-k}i m^{i}E_{i}(0)H_{k+i}(x)\\
&=\sum_{i=0}^{\infty}(-1)^{i}\binom{k+i-1}{k-1}m^{i}E_{i}(0)H_{k+i}(x).
\end{aligned}$$

(4) We have
$$\begin{aligned}
v_{p}(E_{k,\chi})&=v_{p}\left(N^{k}\sum_{a=1}^{N}(-1)^{a}\chi(a)E_{k}(0)\left(\frac{a}{N}\right)\right)\\
&\geq\min\limits_{1\leq a\leq N}v_{p}\left(N^{k}(-1)^{a}\chi(a)E_{k}(0)\left(\frac{a}{N}\right)\right)\\
&=\min\limits_{1\leq a\leq N}v_{p}\left(N^{k}(-1)^{a}\chi(a)\sum_{j=0}^{k}\binom{k}j\left(\frac{a}{N}\right)^{j}E_{k-j}(0)\right)\\
&\geq\min\limits_{\substack{1\leq a \leq N \\0\leq j\leq k}}v_{p}\left(N^{k}(-1)^{a}\chi(a)\binom{k}{j}\left(\frac{a}{N}\right)^{j}E_{k-j}(0)\right)\\
&\geq 0
\end{aligned}$$
proving (4).
\end{proof}

\section*{Acknowledgment}

This work was supported by the National Research Foundation of
Korea(NRF) grant funded by the Korea government(MEST) (2011-0001184
). The authors are enormously grateful to the anonymous referee
whose comments and suggestions lead to a large improvement of the
paper.

\bibliography{central}

\begin{thebibliography}{00}

\bibitem{Ayoub}
R. Ayoub, \textit{Euler and Zeta Function}, The American
Mathematical Monthly \textbf{81} (1974), 1067--1086.

\bibitem{CE}
K.-W. Chen and M. Eie, \textit{A note on generalized Bernoulli
numbers}, Pacific J. Math. \textbf{199} (2001), 41--59.

\bibitem{Ca}
L. Carlitz, \textit{The class number of an imaginary quadratic
field},  Comment. Math. Helv. \textbf{27} (1953), 338--345.

\bibitem{Co1}
H. Cohen, \textit{Number theory Vol. I: Tools and Diophantine
equations}, Graduate Texts in Mathematics, 239. Springer, New York,
2007.

\bibitem{Co2}
H. Cohen, \textit{Number Theory Vol. II: Analytic and Modern Tools},
Graduate Texts in Mathematics, 240. Springer, New York, 2007.

\bibitem{CF}
H. Cohen and E. Friedman, \textit{Raabe's formula for $p$-adic gamma
and zeta functions}, Ann. Inst. Fourier (Grenoble) \textbf{58}
(2008), 363--376.

\bibitem{Ge}
I. M. Gessel, \textit{On Miki's identity for Bernoulli numbers}, J.
Number Theory \textbf{110} (2005), 75--82.

\bibitem{Iw} K. Iwasawa,
\textit{Lectures on $p$-Adic $L$-Functions}, Ann. of Math. Stud. 74,
Princeton Univ. Press, Princeton, 1972.

\bibitem{Ka} K. Kamano,
\textit{$p$-adic $q$-Bernoulli numbers and their denominators}, Int.
J. Number Theory \textbf{4} (2008), no. 6, 911--925.

\bibitem{KKW} M. Kaneko, N. Kurokawa and M. Wakayama,
\textit{A variation of Euler's approach to values of the Riemann
zeta function}, Kyushu J. Math. \textbf{57} (2003), 175--192.

\bibitem{Kashio}T.Kashio, \textit{On a $p$-adic analogue of Shintani's
formula}, J. Math. Kyoto Univ. \textbf{45} (2005),  99--128.


\bibitem{KK}T. Kashio and H. Yoshida,\textit{On $p$-adic absolute
CM-periods}, I. Amer. J. Math. \textbf{130} (2008), 1629--1685.

\bibitem{MSK}
M.-S. Kim, \textit{On Euler numbers, polynomials and related
$p$-adic integrals}, J. Number Theory \textbf{129} (2009),
2166--2179.

\bibitem{MSK2}
M.-S. Kim, \textit{On the behavior of $p$-adic Euler $L$-functions},
arXiv:1010.1981.

\bibitem{TK99} T. Kim,
\textit{On a $q$-analogue of the $p$-adic log gamma functions and
related integrals}, J. Number Theory \textbf{76} (1999), 320--329.

\bibitem{TK}
T. Kim, \textit{On the analogs of Euler numbers and polynomials
associated with $p$-adic $q$-integral on $\mathbb Z_p$ at $q=-1$},
J. Math. Anal. Appl. \textbf{331} (2007), 779--792.

\bibitem{TK08}
T. Kim, \textit{Euler numbers and polynomials associated with Zeta
functions}, Abstract and Applied analysis. Article ID 581582, 2008.

\bibitem{Ko} N. Koblitz,
\textit{$p$-adic Numbers, $p$-adic Analysis and Zeta-Functions}, 2nd
ed., Springer-Verlag, New York, 1984.

\bibitem{KL}
T. Kubota and H. W. Leopoldt, \textit{Eine $p$-adische Theorie der
Zetawerte I, Einfuhrung der $p$-adischen Dirichletschen
$L$-Funktionen}, J. Reine Angew. Math. {\bf 214/215} (1964),
328--339.

\bibitem{La} S. Lang,
\textit{Cyclotomic Fields I and II}, Combined 2nd ed.,
Springer-Verlag, New York, 1990.

\bibitem{Ma}H. Ma\"iga,
\textit{Some identities and congruences concerning Euler numbers and
polynomials}, J. Number Theory \textbf{130} (2010), 1590--1601.

\bibitem{Murty}
M. R. Murty and M. Reece, \textit{A simple derivation of
$\zeta(1-k)=-B_{K}/K$}, Functiones et Approximatio. \textrm{XXVIII}
(2000), 141--154.

\bibitem{Osipov}
Ju. V. Osipov, \textit{p-adic zeta functions}, (Russian), Uspekhi
Mat. Nauk  \textbf{34} (1979),  209--210.

\bibitem{Ra}
M. Ram Murty, \textit{Introduction to $p$-adic analytic number
theory}, AMS/IP Studies in Advanced Mathematics, 27. American
Mathematical Society, Providence, RI; International Press,
Somerville, MA, 2002.

\bibitem{Ro}
A. M. Robert, \textit{A course in $p$-adic analysis}, Graduate Texts
in Mathematics, 198, Springer-Verlag, New York, 2000.

\bibitem{SC} W. H. Schikhof, \textit{Ultrametric Calculus. An Introduction to p-Adic
Analysis}, Cambridge University Press, London, 1984.

\bibitem{Shi} K. Shiratani and S. Yamamoto, \textit{On a p-adic
interpolation function for the Euler numbers and its derivatives},
Mem. Fac. Sci. Kyushu Univ. Ser. A \textbf{39} (1985), 113--125.

\bibitem{Su}
Z.-W. Sun, \textit{On Euler numbers modulo powers of two}, J. Number
Theory \textbf{115} (2005), 371--380.

\bibitem{TP}
B. A. Tangedal and P. T. Young, \textit{On $p$-adic multiple zeta
and log gamma functions}, J. Number Theory \textbf{131} (2011),
1240--1257.

\bibitem{Vo}
A. Volkenborn, \textit{Ein $p$-adisches Integral und seine
Anwendungen I}, Manuscripta Math. \textbf{7} (1972), 341--373.

\bibitem{Vo1}
A. Volkenborn, \textit{Ein $p$-adisches Integral und seine
Anwendungen II}, Manuscripta Math. \textbf{12} (1974), 17--46.

\bibitem{Jr} S. S. Wagstaff, Jr, \textit{Prime divisors of the Bernoulli and Euler numbers},
Number theory for the millennium, III (Urbana, IL, 2000), 357--374,
A K Peters, Natick, MA, 2002.

\bibitem{Wa} L. C. Washington,
\textit{Introduction to Cyclotomic Fields}, 2nd ed.,
Springer-Verlag, New York, 1997.

\bibitem{xzhang} X. Zhang, \textit{Ten Formulae of Type Ankeny-Artin-Chowla for
Class Numbers of General Cyclic Quartic Fields}, Scientia
Sinica(A), \textbf{32} (1989),417--428;  Math. Rev. 91b: 11112.

\end{thebibliography}

\end{document}